\documentclass[12pt,reqno]{amsart}

\usepackage[all]{xy}
\usepackage{hyperref}

\usepackage{amsmath,amssymb,amsfonts,amsthm,mathrsfs, mathtools}

\newcommand{\R}{\mathbb{R}}

\newcommand{\Z}{\mathbb{Z}}

\newcommand{\C}{\mathbb{C}}

\renewcommand{\O}{\mathcal{O}}

\renewcommand{\O}{\mathcal{O}}

\renewcommand{\P}{\mathbb{P}}

\newcommand{\dbar}{\overline{\partial}}
\newcommand{\NN}{\mathcal{N}}

\renewcommand{\a}{\mathfrak{a}}

\newcommand{\g}{\mathfrak{g}}
\newcommand{\h}{\mathfrak{h}}

\renewcommand{\k}{\mathfrak{k}}
\renewcommand{\l}{\mathfrak{l}}

\newcommand{\p}{\mathfrak{p}}

\renewcommand{\u}{\mathfrak{u}}

\renewcommand{\sl}{\mathfrak{sl}}
\renewcommand{\t}{\mathfrak{t}}

\newcommand{\w}{\mathfrak{w}}

\newcommand{\dd}{\frac{d}{d\epsilon} \bigg|_{\epsilon = 0}}

\DeclareMathOperator{\Hom}{Hom}

\DeclareMathOperator{\height}{ht}
\DeclareMathOperator{\ad}{ad}
\DeclareMathOperator{\At}{At}

\DeclareMathOperator{\End}{End}
\DeclareMathOperator{\Tot}{Tot}

\newcommand{\Ad}{\textnormal{Ad}}

\newcommand{\Id}{\textnormal{Id}}

\newcommand{\commsq}[8]{\xymatrix{ #1 \ar[r]^{#5} \ar[d]_{#6} & #2 \ar[d]^{#7} \\ #3 \ar[r]_{#8} & #4 }}
\newcommand{\m}{\mathfrak{m}}

\newcommand{\vv}{\mathfrak{v}}

\newtheorem{thm}[equation]{Theorem}
\newtheorem{lem}[equation]{Lemma}
\newtheorem{prop}[equation]{Proposition}
\newtheorem{cor}[equation]{Corollary}

\theoremstyle{definition}
\newtheorem{defn}[equation]{Definition}
\newtheorem{rmk}[equation]{Remark}

\numberwithin{equation}{section}

\begin{document}

\title[Connection for bundles on $G/P$ and twistor space of coadjoint 
orbits]{The universal connection for principal bundles over 
homogeneous spaces and twistor space of 
coadjoint orbits}

\author[I.~Biswas]{Indranil Biswas}

\address{School of Mathematics, Tata Institute of Fundamental Research, Homi Bhabha Road, Mumbai 400005, India}

\email{indranil@math.tifr.res.in}

\author[M.L.~Wong]{Michael Lennox Wong}

\address{Universit\"at Duisburg-Essen, Fakult\"at f\"ur Mathematik, Thea-Leymann-Str.~9, 45127 Essen, Germany}

\email{michael.wong@uni-due.de} 



\begin{abstract}
Given a holomorphic principal bundle $Q\, \longrightarrow\, X$, the universal space of holomorphic connections is a torsor 
$C_1(Q)$ for $\ad Q \otimes T^*X$ such that the pullback of $Q$ to $C_1(Q)$ has a tautological holomorphic connection. When 
$X\,=\, G/P$, where $P$ is a parabolic subgroup of a complex simple group $G$, and $Q$ is the frame bundle of an ample line 
bundle, we show that $C_1(Q)$ may be identified with $G/L$, where $L\, \subset\, P$ is a Levi factor. We use this identification 
to construct the twistor space associated to a natural hyper-K\"ahler metric on $T^*(G/P)$, recovering Biquard's description of 
this twistor space, but employing only finite-dimensional, Lie-theoretic means.
\end{abstract}

\maketitle

\tableofcontents

\section{Introduction}

Let $X$ be a complex manifold, $G$ a complex Lie group and $Q$ a holomorphic principal $G$-bundle over $X$. As is well-known, 
holomorphic connections on $Q$ may be identified with holomorphic splittings of the Atiyah sequence canonically associated to 
$Q$:
$$
0\, \longrightarrow\, \ad Q\, \longrightarrow\, \At Q \, \longrightarrow\, \Theta_X \, \longrightarrow\, 0;
$$
here, $\Theta_X$ is the holomorphic tangent bundle to $X$ and $\At Q$ is the Atiyah bundle \cite[\S~2]{Atiyah}. It is possible 
that while there is no such splitting over $X$, upon pullback via some holomorphic mapping $f \,:\, Y \,\longrightarrow\, X$, 
where $Y$ is another complex manifold, this sequence does split. In fact, a tautological construction yields such a $Y$ for any 
$Q$, which can be briefly described as follows. Tensoring the Atiyah exact sequence with the holomorphic cotangent bundle 
$\Omega_X\,=\, \Theta_X^*$ we get
$$
0\, \longrightarrow\, \ad Q \otimes \Omega_X\, \longrightarrow\, \At Q \otimes \Omega_X \, \stackrel{s}{\longrightarrow}\, \Theta_X \otimes \Omega_X\,=\, \End(\Theta_X) \, \longrightarrow\, 0\, .
$$
The inverse image $C_1(Q)\,:=\, s^{-1}(\text{Id}_{\Theta_X})\, \subset\, \At Q \otimes \Omega_X\, \longrightarrow\, X$ is the universal space of holomorphic connections in the sense that for any open subset $U\, \subset\, X$, the holomorphic sections of the fibre bundle $C_1(Q)\vert_U$ over $U$ are precisely the holomorphic connections on $Q\vert_U$. The details of this construction as well as other related results are found in Sections \ref{s:cxnpb} and \ref{s:unpb}. In fact, since it presents no further complication and some statements prove to be useful later, we work with $\lambda$-connections for $\lambda \in \C$.

In Section \ref{s:Atseqhomogsp}, we consider the case where $X$ is a (complex) homogeneous space. In this case, the sequences of vector bundles above all have descriptions as those associated to canonical sequences of representations for the groups involved. Of course, the universal pullback connection space also has such a simple description, which we give. Emphasis is on the case where $P$ is a parabolic subgroup of complex simple affine algebraic group $G$, so that $X = G/P$ is a projective rational homogeneous space. Let $Q$ be a holomorphic $\C^\times$-bundle over $G/P$ associated to a strictly anti-dominant character of $P$, so that the associated line bundle is (very) ample. If $L\, \subset\, P$ is a Levi factor, then the main result of Section \ref{s:Atseqhomogsp} shows that the above fibre bundle $C_1(Q)\, \longrightarrow\, G/P$ may be identified with the canonical projection $G/L\,\longrightarrow\, G/P$. In particular, $G/L$ may be viewed as a complexification of $G/P$.

We recall that a hyper-K\"ahler manifold is a $C^\infty$ manifold $M$ equipped with
\begin{itemize}
\item integrable almost complex structures, $I$, $J$ and $K$ satisfying the quaternionic relation $IJK \,=\, -\text{Id}$, and

\item a Riemannian metric $g$ which is K\"ahler with respect to each of $I$, $J$ and $K$.
\end{itemize}
To a hyper-K\"ahler manifold there is an associated twistor space, which is a holomorphic fibre bundle
$$
Z\,\longrightarrow\, \P^1\, =\, S^2\,=\, \{(a,\, b,\, c)\, \in\, {\mathbb R}^3\, \mid\, a^2+b^2+c^2\,=\,1\}
$$
such that the fibre over $(a,\, b,\, c)\, \in\, S^2$ is $Y$ equipped with the integrable almost complex structure $aI+bJ+cK$. Thus, any fibre over $\P^1$ may be identified with the original $C^\infty$ manifold. Furthermore, it is a fundamental theorem that from such a fibre bundle over $\P^1$, if one also has a compatible real structure and fibre-wise holomorphic symplectic form, one can recover the hyper-K\"ahler metric on the fibres \cite[Theorem 3.3]{HKLR}.

Hyper-K\"ahler metrics on coadjoint orbits for semisimple algebraic 
groups were first constructed by Kronheimer \cite{Kronheimer} in the 
case of regular semisimple orbits; this was generalized to arbitrary 
orbits by Biquard \cite{Biquard1}. Furthermore, Biquard also gave a 
description of the twistor space in these cases, in which the general 
fibre is (isomorphic to) the coadjoint orbit and the special fibres 
(which one usually pictures over $0, \infty \in \P^1$) are the cotangent 
bundles to a (partial) flag variety \cite{Biquard1,Biquard2}. The 
technical method employed in this series of papers is the use of Nahm's 
equations, the hyper-K\"ahler metric thus arising as an 
infinite-dimensional hyper-K\"ahler quotient.

Of course, given a semisimple element in a semisimple complex Lie 
algebra, its (co)adjoint stabilizer is a reductive subgroup of $G$ of 
full rank, hence may be understood as a Levi factor $L$ of some 
parabolic subgroup $P$ of $G$; therefore the coadjoint orbit is 
isomorphic to $G/L$. Using the constructions of Section 
\ref{s:Atseqhomogsp}, in Section \ref{s:twistor}, we are able to recover 
the construction of the twistor space given by Biquard, thereby 
obtaining the existence of a hyper-K\"ahler metric on $T^*(G/P)$ or, 
equivalently, $G/L$---via the fundamental theorem mentioned 
above---using only the means of Lie theory.

The method we use to obtain the existence of the hyper-K\"ahler metric on $T^*(G/P)$---namely, by constructing the twistor space directly---resembles that of Feix, who showed that, for a K\"ahler manifold $X$, some neighbourhood of the zero section in $T^*X$ always admits a hyper-K\"ahler metric \cite[Theorem A]{Feix} (the same result was obtained by Kaledin, but by different methods \cite[Theorem 1.1]{Kaledin}). In the cases we consider, we find that the hyper-K\"ahler metric in fact exists on the entirety of the cotangent bundle, which does not necessarily hold in general (cf.~\cite[Theorem B]{Feix}).

There exist some hyper-K\"ahler moduli spaces, prominent among them the character variety for a compact Riemann surface, which have elementary finite-dimensional constructions, but for which there is no known finite-dimensional description or construction of the metric. (For the example of the character variety just mentioned, the construction is as an
affine geometric invariant theory quotient, yet the hyper-K\"ahler metric is by an infinite-dimensional hyper-K\"ahler
quotient, via a dimensional reduction of the Yang--Mills equations \cite[\S~6]{Hitchin}.) The result here may be regarded as a step towards a finite-dimensional understanding of these metrics.

\section{Principal bundles, connections and pullbacks} \label{s:cxnpb}

Let $X$ be a connected complex manifold; its holomorphic tangent and cotangent bundles will be denoted by $\Theta_X$ and $\Omega_X$, respectively. We will write $T^*X \,:= \,\Tot(\Omega_X)$ for the total space of $\Omega_X$. Let $G$ be a complex Lie group; its Lie algebra will be denoted by $\g$. Let $\pi \,:\, P \,\longrightarrow\, X$ be a holomorphic principal $G$-bundle over $X$ with $G$ acting on the right of $P$. The holomorphic tangent bundle of the total space of $P$ will be denoted by $\Theta_P$. The adjoint vector bundle $\ad P\,=\, P\times^G \g$ is the one associated to $P$ for the adjoint action of $G$ on $\g$. The pullback $\pi^*\ad P$ is the trivial vector bundle $\O_P^\g\,=\, P\times{\mathfrak g}\, \longrightarrow\, P$, which, in turn, is identified with $\ker(d\pi)$ by the action of $G$ on $P$, where $d\pi\, :\, \Theta_P\,\longrightarrow\, \pi^*\Theta_X$ is the differential of the projection $\pi$. Let
$$
\At P\, :=\, \Theta_P/G\,=\, (\pi_*\Theta_P)^G\, \subset\, \pi_*\Theta_P
$$
be the Atiyah bundle and
\begin{align} \label{A}
0 \,\longrightarrow\, \ad P \,\longrightarrow\, \At P \,\longrightarrow\, \Theta_X \,\longrightarrow\, 0
\end{align}
the Atiyah exact sequence for $P$, which is the quotient by $G$ of the exact sequence
\begin{align}\label{A1}
0\,\longrightarrow\, \pi^*\ad P\,=\, \ker(d\pi) \,\stackrel{\iota}{\longrightarrow}\, \pi^*\At P\,=\, \Theta_P \,\stackrel{d\pi}{\longrightarrow}\, \pi^* \Theta_X \,\longrightarrow\, 0\, 
\end{align}
on $P$. To describe $\At P$ in terms of local trivializations, fix a holomorphically trivializing open cover $\{ X_\alpha \}$ of $X$, so that there exist $G$-equivariant holomorphic maps 
$$\varphi_\alpha \,:\, P_\alpha \,:=\, P|_{X_\alpha} \,
\stackrel{\sim}{\longrightarrow}\,X_\alpha \times G$$ 
over the identity map of $X_\alpha$. Let $g_{\alpha \beta} \,:\, X_{\alpha \beta}\,:=\, X_\alpha \cap X_\beta\,\longrightarrow\, G$ be the corresponding transition functions satisfying
\begin{align*}
\varphi_\alpha \circ \varphi_\beta^{-1}(x,\, g) \,=\, \big( x,\, g_{\alpha \beta}(x) a \big)\, .
\end{align*}
Then one has induced isomorphisms $\phi_\alpha \,:\, \At P_\alpha \,
\stackrel{\sim}{\longrightarrow} \,\Theta_\alpha \oplus \O_\alpha^\g$, where $\Theta_\alpha \,:= \,\Theta_{X_\alpha} \,=\, \Theta_X|_{X_\alpha}$ and $\O_\alpha^\g \,:=\, \O_{X_\alpha} \otimes_\C \g$, for which
\begin{align} \label{B}
\phi_\alpha \circ \phi_\beta^{-1}(v, \,\xi) \,=\, \left( v,\, \Ad_{g_{\alpha \beta}} \xi + dg_{\alpha \beta} g_{\alpha \beta}^{-1}(v) \right).
\end{align}

Let $\lambda \in \C$. We recall that a holomorphic $\lambda$-connection on $P$ is an ${\mathcal O}_X$-linear homomorphism $s\, :\, \At P \,\longrightarrow\, \ad P$ whose composition with the inclusion of \eqref{A}
\begin{align*}
\ad P \,\longrightarrow\, \At P \,\stackrel{s}{\longrightarrow}\, \ad P
\end{align*}
is simply multiplication by the scalar $\lambda$ on $\ad P$; in the case $\lambda \,=\, 1$ this $s$ is a holomorphic splitting of \eqref{A}, and so a holomorphic connection on $P$ in the usual sense. A $0$-connection is a homomorphism $\Theta_X \,\longrightarrow\,\ad P$. Note that a $\lambda$-connection may equivalently be described as a $G$-equivariant homomorphism $\widetilde{s}\, :\,\Theta_P \,\longrightarrow\, \O_P^\g$ such that $\widetilde{s}\circ\iota\,=\, \lambda\cdot\text{Id}_{\ad P}$, where $\iota$ is the homomorphism in \eqref{A1}. Such a homomorphism $\widetilde{s}$ defines a $G$-equivariant $\g$-valued holomorphic $1$-form on $P$ (the group $G$ has the adjoint action on $\mathfrak g$); it is called the $1$-form of the $\lambda$-connection. The kernel of a connection $1$-form $\Theta_P \,\longrightarrow\, \O_P^\g$ is called the horizontal distribution of the connection.

Suppose $Y$ is another complex manifold and $f\,:\, Y \,\longrightarrow\, X$ a holomorphic map. We may pull back the Atiyah sequence \eqref{A} along $f$ to get an exact sequence on $Y$. Also, we may pull back $P$ along $f$ to get a principal $G$-bundle $f^*P$ on $Y$, which has its own Atiyah sequence; one then has a morphism between these exact sequences
\begin{equation} \label{B1}
\vcenter{ \xymatrix{
0 \ar[r] & \ad f^* P \ar[r] \ar@{=}[d] & \At f^*P \ar[r] \ar[d]^\beta
& \Theta_Y \ar[r] \ar[d]^{df} & 0 \\
0 \ar[r] & f^* \ad P \ar[r] & f^*\At P \ar[r] & f^* \Theta_X 
\ar[r] & 0.
} }
\end{equation}
The above homomorphism $\beta$ is constructed as follows: consider the Cartesian diagram
\begin{align*}
\commsq{ f^*P }{ P }{ Y }{ X }{ F }{ \Pi }{ \pi }{ f }
\end{align*}
associated to the above pair $(P,\, f)$. This produces a commutative diagram on the total space $f^*P$
\begin{align} \label{B2}
\vcenter{ \xymatrix{
0 \ar[r] & \O_{f^* P}\otimes\g \ar[r] \ar@{=}[d] & \Theta_{f^*P} \ar[r] \ar[d]^{dF} & \Pi^* \Theta_Y \ar[r] \ar[d]^{\Pi^*df} & 0 \\
0 \ar[r] & \O_{f^* P}\otimes\g \ar[r] & F^* \Theta_P \ar[r] & \Pi^* f^* \Theta_X = F^* \pi^* \Theta_X \ar[r] & 0. } }
\end{align}
Since the differential $dF$ in \eqref{B2} is $G$-equivariant, the diagram in \eqref{B2} descends to a commutative diagram of homomorphisms on $Y$. This descended diagram is the one in \eqref{B1}. Therefore, the diagram in \eqref{B2} is the pullback, by the map $\Pi$, of the diagram in \eqref{B1}. Note that if $f$ is a submersion, meaning $df$ is surjective, then $\beta$ is surjective as well. If $$s\, :\, \At P \,\longrightarrow\, \ad P$$ is a $\lambda$-connection on $P$, then $(f^*s) \circ \beta \,:\, \At f^* P \,\longrightarrow\, \ad f^* P$ is a $\lambda$-connection on $f^*P$. In other words, a $\lambda$-connection on $P$ pulls back to a $\lambda$-connection on $f^*P$.

\begin{defn} \label{d:trivfib}
Let $X$, $Y$, $P$ and $f \,:\, Y\,\longrightarrow\, X$ be as above, and let $D$ be a holomorphic $\lambda$-connection on the pulled back principal bundle $f^* P$ given by a splitting $s \,:\, \At f^* P \,\longrightarrow\, \ad f^*P$. We say that $D$ is \emph{trivial on the fibres of $f$} if there is a homomorphism $$s' \,:\, f^*\At P \,\longrightarrow\, f^*\ad P$$ such that $s \,=\, s'\circ\beta$, where $\beta$ is the homomorphism in \eqref{B1}.
\end{defn}

It should be clarified that the above condition does not mean that $D$ is the pullback of a $\lambda$-connection on $P$. More precisely, $s'$ need not be the pullback of a splitting of the Atiyah exact sequence for $P$.

Lemma \ref{lemma1} is straightforward to prove, so we omit its proof.

\begin{lem}\label{lemma1}
The following are equivalent.
\begin{enumerate}
\item[(a)] $D$ is trivial on the fibres of $f$.

\item[(b)] The connection $1$-form associated to $D$, which is a section of $\Omega_{f^*P} \otimes \g$, is a section of the sub-sheaf $(dF)^*(F^*\Omega_P) \otimes \g$ of $\Omega_{f^*P} \otimes \g$, where $(dF)^*$ is the dual of the homomorphism $dF$ in \eqref{B2}.

\item[(c)] The horizontal distribution for $D$ contains the relative tangent bundle ${\rm ker}(dF)\, \subset\, Tf^*P$ for $F$.

\item[(d)] For any open cover $\{ X_\alpha \}$ of $X$ with holomorphic trivializations of $P\vert_{X_\alpha}$, if
$$D_\alpha \,\in\, \Gamma(f^{-1}(X_\alpha), \,\Omega_Y \otimes \g)$$ 
are the connection $1$-forms on $f^{-1}(X_\alpha)$ associated to the corresponding trivializations of $f^*P$ over $\{(f^{-1}(X_\alpha)\}$, then
\begin{equation}\label{coc}
D_\alpha \,\in\, \Gamma(f^{-1}(X_\alpha),\, (df)^* (f^*\Omega_X) \otimes \g)\, ,
\end{equation}
where $(df)^*\, :\, f^*\Omega_X\,\longrightarrow\, \Omega_Y$ is the dual of the differential of $f$.

\item[(e)] There exists a $P$ trivializing open cover $\{ X_\alpha \}$ of $X$ as above such that \eqref{coc} holds.
\end{enumerate}
\end{lem}

\section{The universal pullback $\lambda$-connection}\label{s:unpb}

Let $X$ be as before.

\subsection{Construction and properties of the universal pullback $\lambda$-connection} \label{s:univlamcxn}

The following can be seen by taking \v{C}ech or Dolbeault representatives for cohomology.

\begin{lem}\label{1}
Let
\begin{align} \label{C}
0 \,\longrightarrow\, U \,\longrightarrow\, V \,
\stackrel{\sigma}{\longrightarrow}\, W \,\longrightarrow\, 0
\end{align}
be a short exact sequence of holomorphic vector bundles over $X$. Let $s \,\in\, H^0(X,\, W)$ and $V_s \,:=\, \sigma^{-1}( s(X))$. Then $V_s\,\longrightarrow\, X$ is an affine bundle for $U$. Furthermore, such bundles are classified by $H^1(X,\, U)$: if $\alpha \,\in\, H^1(X, \,U\otimes W^\vee)$ is the extension class of \eqref{C}, then $V_s$ corresponds to $\langle \alpha,\, s \rangle\,\in\, H^1(X,\, U)$, where $\langle \, , \rangle \,:\, H^1(X, \,U\otimes W^\vee) \otimes H^0(X, \,W) \,\longrightarrow\, H^1(X,\, U)$ is the homomorphism induced by the evaluation homomorphism $(U\otimes W^\vee)\otimes W\,\longrightarrow\, U$.
\end{lem}

Again, let $P$ be a holomorphic principal $G$-bundle over $X$. Apply Lemma \ref{1} to the Atiyah sequence \eqref{A} tensored by $\Omega_X$; the last term will be $\Theta_X \otimes \Omega_X = \End \Omega_X$, which has a sub-bundle $\O_X \hookrightarrow \End \Omega_X$ defined by 
\begin{align*}
f \,\longmapsto\, f \cdot{\rm Id}_{\Omega_X}\, .
\end{align*}
We let $W \,:=\, W_P$ be its preimage under the surjection $\At P \otimes \Omega_X \,\longrightarrow\, \End \Omega_X$. Hence we get a diagram
\begin{align} \label{e:AtOmega}
\vcenter{ \xymatrix{
0 \ar[r] & \ad P \otimes \Omega_X \ar[r] \ar@{=}[d] & \At P \otimes \Omega_X \ar[r]^\sigma & \End \Omega_X \ar[r] & 0 \\
0 \ar[r] & \ad P \otimes \Omega_X \ar[r] & W_P \ar[r]^\sigma \ar@{^{(}->}[u] & \O_X \ar[r] \ar@{^{(}->}[u] & 0.
} }
\end{align}
Now set
\begin{equation}\label{wp}
Z^\circ \,=\, Z_P^\circ \,:=\, \Tot(W_P)\, .
\end{equation}
Although the restriction
of $\sigma$ to $W_P$ is also denoted by $\sigma$, this should not cause any confusion.

Using the expressions in \eqref{B}, there exist isomorphisms $$\widehat{\phi}_\alpha \,:\, (\At P \otimes \Omega)|_{X_\alpha} \,\stackrel{\sim}{\longrightarrow}\, (\Omega_\alpha \otimes \g) \oplus \End \Omega_\alpha, $$ 
where $\Omega_\alpha := \Omega_X|_{X_\alpha}$, such that
\begin{align*}
\widehat{\phi}_\alpha \circ \widehat{\phi}_\beta^{-1}(\xi, \gamma) \,= \,\left( \Ad_{g_{\alpha \beta}} \xi + \gamma( dg_{\alpha \beta} g_{\alpha \beta}^{-1} ), \gamma \right);
\end{align*}
where $\gamma \in \End \Omega_X$ is applied to the $\Omega_X$-factor of $dg_{\alpha \beta} g_{\alpha \beta}^{-1}$. Similarly, one has isomorphisms for $W$:
\begin{align} \label{e:Wtriv}
\widetilde{\phi}_\alpha & \,:\, W|_{X_\alpha} \,\stackrel{\sim}{\longrightarrow}\,
(\Omega_\alpha \otimes \g) \oplus \O\, , & \widetilde{\phi}_\alpha \circ \widetilde{\phi}_\beta^{-1}( \xi, \lambda) \,= \,\left( \Ad_{g_{\alpha \beta}} \xi + \lambda \cdot dg_{\alpha \beta} g_{\alpha \beta}^{-1} , \lambda \right)\, .
\end{align}

Let $q \,:\, Z^\circ \,\longrightarrow\, X$ be the projection
(see \eqref{wp}), and set $Z_\alpha^\circ \,:=\,
q^{-1}(X_\alpha)$. One also has a projection $\pi^\circ \,=\, \lambda \,:\, Z^\circ
\,\longrightarrow\, \C$ given by the composition
\begin{align*}
Z^\circ \,=\, \Tot(W_P) \,\stackrel{\sigma}{\longrightarrow}\, \Tot(\O_X) \,=\, X \times \C
\,\longrightarrow\, \C\, .
\end{align*}
(The reason for the notation $\pi^\circ \,=\, \lambda$ should be made clear through our usage here
and later in Section \ref{s:twistor}.) For each $\alpha$, write $q_\alpha \,:\, \Tot(\Omega_\alpha
\otimes \g) \,\longrightarrow\, X_\alpha$ for the projection. We also obtain projection maps 
\begin{align*}
\Tot\big( (\Omega_\alpha \otimes \g) \oplus \O \big) \,\cong\, \Tot(\Omega_\alpha \otimes \g)
\times \C \,\longrightarrow\, \Tot(\Omega_\alpha \otimes \g)\, .
\end{align*}
Combining these with the isomorphisms induced by \eqref{e:Wtriv}, we get a diagram
\begin{align} \label{e:Zalphatot}
\vcenter{ \xymatrix{
Z_\alpha^\circ \ar[dr]_q \ar[r]^-{\widetilde{\phi}_\alpha} 
& \Tot\left( (\Omega_\alpha \otimes \g) \oplus \O \right) \ar[d] \ar[r] & \Tot(\Omega_\alpha \otimes \g) \ar[dl]^{q_\alpha} \\
& X_\alpha & 
} }
\end{align}
In \eqref{e:Zalphatot}, the composition across the top will be denoted by $\rho_\alpha$.
Let $\widetilde{q} \,:\, \Tot(\Omega_X \otimes \g) \,\longrightarrow\, X$ be the natural
projection from the total space of $\Omega_X \otimes \g$. There is a tautological section
\begin{align*}
a \,\in\, \Gamma( \Tot(\Omega_X \otimes \g),\, \widetilde{q}^*(\Omega_X \otimes \g))\, ,
\end{align*}
which one may of course restrict to each $\widetilde{q}^{-1}(X_\alpha) \,=\,
\Tot(\Omega_\alpha \otimes \g)$. Then set
\begin{align} \label{e:lambdacxnform}
A_\alpha \,:=\, \rho_\alpha^* a \in \Gamma(Z_\alpha^\circ, \,q^* \Omega_\alpha \otimes \g)
\,\subseteq \,\Gamma(Z_\alpha^\circ,\, \Omega_{Z_\alpha^\circ} \otimes \g)\, .
\end{align}
These satisfy
\begin{align} \label{e:lambdacxntfm}
A_\alpha \,=\, g_{\alpha \beta} A_\beta g_{\alpha \beta}^{-1} - \lambda \, dg_{\alpha \beta}
\cdot g_{\alpha \beta}^{-1}\, ,
\end{align}
where $g_{\alpha \beta}$ of course means $q^* g_{\alpha \beta}$ and $\lambda \,:\,
Z_\alpha^\circ \,\longrightarrow\, \C$ is the projection described above.

For $\lambda \,\in \,\C$, consider the constant section 
$\lambda \,\in \,H^0(X,\, \O_X)$ (this may be identified with $\lambda \cdot \Id_{\Omega_X} \,\in\,
H^0(X,\, \End \Omega_X)$ in the diagram \eqref{e:AtOmega}). We then set
\begin{align} \label{e:ClamDef}
C_\lambda(P) \,:= \, (\pi^\circ)^{-1}(\lambda) \, \subset\, Z^\circ\, ,
\end{align}
where the right hand side uses the notation of Lemma \ref{1}. If $P$ is understood from the
context, we will shorten this to $C_\lambda$, and if
$\lambda \,=\, 1$ we shorten it further still to $C$. The natural projection
\begin{equation}\label{pi}
q\, :\, C_\lambda \,\longrightarrow\, X\, ,
\end{equation}
which is a surjective submersion, makes $C_\lambda$ an affine bundle over $X$ with bundle of translations $\ad P \otimes \Omega_X$; in the case $\lambda = 0$, one has simply $C_0 \,=
\, \Tot(\ad P \otimes \Omega_X)$. The space $C_\lambda(P)$ is the base of the \emph{universal pullback $\lambda$-connection} on $P$ for the following reason.

\begin{prop}\label{prop1}
Fix $\lambda \,\in\, \C$. The space $q \,:\, C_\lambda(P) \,\longrightarrow\, X$ is an affine bundle over $X$ with $\ad P \otimes \Omega_X$ as its bundle of translations. The pulled back bundle $q^*P$ under the map in \eqref{pi} admits a canonical $\lambda$-connection $\nabla$ which is trivial on the fibres of $q$. Furthermore, it has the universal property that if $f \,:\, Y \,\longrightarrow\, X$ is any smooth morphism such that $f^*P$ admits a $\lambda$-connection $D$ trivial on the fibres of $f$, then there exists a unique morphism $\psi \,:\, Y \,\longrightarrow\, C_\lambda$ such that
\begin{align*}
\xymatrix{ Y \ar[r]^\psi \ar[dr]_f & C_\lambda \ar[d]^q \\ & X }
\end{align*}
commutes and $(f^*P,\, D) \,\cong\, \psi^*( q^* P,\, \nabla)$.
\end{prop}

\begin{proof}
The first statement is Lemma \ref{1}.

As for $Z^\circ$, we write $C_{\lambda, \alpha} \,:=\, q^{-1}(X_\alpha) \cap C_\lambda$. Since $C_\lambda \,= \,(\pi^\circ)^{-1}(\lambda)$, the isomorphisms $\widetilde{\phi}_\alpha$ in \eqref{e:Zalphatot} restrict to isomorphisms $C_{\lambda, \alpha}\,\longrightarrow\, \Tot(\Omega_\alpha \otimes \g) \times \{ \lambda \} \,=\, \Tot(\Omega_\alpha \otimes \g)$, which we also denote by $\rho_\alpha$.

The local $1$-forms for the $\lambda$-connection are of course the $A_\alpha$ defined in \eqref{e:lambdacxnform} restricted to $C_{\lambda, \alpha}$; the fact that they yield a well-defined $\lambda$-connection comes from \eqref{e:lambdacxntfm}, noting that now $\lambda$ is a fixed constant. The fact that the resulting $\nabla$ is trivial on the fibres can be seen from \eqref{e:lambdacxnform} and from Lemma \ref{lemma1}.

Suppose now that $f \,:\, Y \,\longrightarrow\, X$ is smooth and there is a $\lambda$-connection 
$D$ on $f^* P$ trivial on the fibres. If $Y_\alpha \,:= \,f^{-1}(X_\alpha)$, then $D$ has 
$\lambda$-connection $1$-forms $D_\alpha\,\in\, \Gamma(Y_\alpha,\, f^* \Omega_X \otimes \g)$. We 
then use the following.

\begin{lem} \label{2}
Let $p\, :\, E\,\longrightarrow\, X$ be a vector bundle and $\sigma \,\in\, H^0( \Tot E,\, p^* E)$ the tautological section. If $f \,:\, Y \,\longrightarrow\, X$ is any morphism and $s \,\in\, H^0(Y,\, f^* E)$, then there exists a unique morphism $\tau \,:\, Y \, \longrightarrow\, \Tot E$ such that 
\begin{align*}
\xymatrix{ & \Tot E \ar[d]^p \\ Y \ar[r]_f \ar[ur]^\tau & X }
\end{align*}
commutes and $s \,=\, \tau^*(\sigma)$.
\end{lem}

Applying this, we get morphisms
$\tau_\alpha \,:\, Y_\alpha \,\longrightarrow\, \Tot(\Omega_\alpha \otimes \g)$ and hence $$\psi_\alpha \,:=\, \rho_\alpha^{-1} \circ \tau_\alpha \,:\, Y_\alpha \,\longrightarrow\, C_{\lambda, \alpha}$$ such that
\begin{align*}
\xymatrix{ Y_\alpha \ar[r]^{\psi_\alpha} \ar[dr]_f & C_{\lambda, \alpha} \ar[d]^q \\ & X }
\end{align*}
commutes. Using the fact that the $D_\alpha$ and $A_\alpha$ both transform as $\lambda$-connection $1$-forms, one then
sees that $\psi_\alpha \,= \,\psi_\beta$ on $Y_\alpha \cap Y_\beta$ and this defines $\psi \,:\, Y\,\longrightarrow\,
C_\lambda$. The uniqueness is clear from the uniqueness statement in Lemma \ref{2}.
\end{proof}

Of course, since $Z^\circ$ is the total space of a vector bundle, there is a scalar multiplication map
\begin{align*}
\C \times Z^\circ \,\longrightarrow\, Z^\circ\, ,
\end{align*}
for which 
\begin{align*}
\pi^\circ (\mu \cdot z) \,=\, \mu \pi^\circ(z)
\end{align*}
for all $\mu\,\in\, \C$, $z\,\in\, Z^\circ$, simply because $\pi^\circ$ is defined as a vector bundle map, hence linear, followed by the projection to the fibre of $\O_X$. Of course, the map restricts to a $\C^\times$ group action
\begin{align*}
\C^\times \times Z^\circ \,\longrightarrow\, Z^\circ\, .
\end{align*}
If we set $Z^\times \,:=\, (\pi^\circ)^{-1}(\C^\times)$ then $Z^\times$ is a $\C^\times$-invariant
open subset of $Z^\circ$ and the action on $Z^\times$ is free. In fact, one has an isomorphism
\begin{align*}
Z^\times \,\stackrel{\sim}{\longrightarrow}\, C \times \C^\times\, ,
\end{align*}
where $C \,=\, (\pi^\circ)^{-1}(1)$, given by
\begin{align} \label{e:gammadef}
z \,\longmapsto\, \big( \pi^\circ(z)^{-1} \cdot z, \pi^\circ(z) \big)
\end{align}
with inverse
\begin{align*}
(y, \,\lambda) \,\longmapsto\, \lambda \cdot y\, .
\end{align*}
In particular, for any $\lambda_1$, $\lambda_2 \in \C^\times$, one has
\begin{align*}
C_{\lambda_1} \,\cong\, C_{\lambda_2}\, .
\end{align*}
The commutativity of the diagram \eqref{e:ZYC} below is clear from the maps just described.

\begin{lem} \label{t:Zstariso}
One has an isomorphism
\begin{align*}
\gamma \,:\, Z^\times \,\stackrel{\sim}{\longrightarrow}\, C \times \C^\times\, ,
\end{align*}
which makes the diagram
\begin{align} \label{e:ZYC}
\vcenter{ \xymatrix{ Z^\times \ar[dr]_{\pi^\circ} \ar[r]^-\gamma & C \times \C^\times \ar[d]^{p_{\C^\times}} \\ & \C^\times } }
\end{align}
commute, where $p_{\C^\times}$ is the projection onto the $\C^\times$ factor.
\end{lem}

\subsection{Further remarks in the case of $\C^\times$-bundles}\label{s:Cstarbdls}

Let $P$ be a holomorphic principal $\C^\times$-bundle on $X$. In this case, $\ad P \,=\,
\O_X$ and hence the lower sequence of \eqref{e:AtOmega} simplifies to
\begin{align} \label{e:Cstarext}
0 \longrightarrow \Omega_X \longrightarrow W_P \longrightarrow \O_X \longrightarrow 0.
\end{align}
As before, $Z^\circ\,:=\, \Tot(W_P)$ and we have a map $\pi^\circ \,:\, Z^\circ \,
\longrightarrow\, \C$ yielding an exact sequence
\begin{align} \label{e:ZrelC}
0 \longrightarrow (\pi^\circ)^* \Omega_\C \longrightarrow \Omega_{Z^\circ} \longrightarrow \Omega_{Z^\circ/\C} \longrightarrow 0,
\end{align}
with $\Omega_{Z^\circ/\C}$ the sheaf of differentials relative to $\pi^\circ$.

The local $1$-forms on $Z_\alpha^\circ$ defined in \eqref{e:lambdacxnform} are now scalar $1$-forms and by \eqref{e:lambdacxntfm} they transform as
\begin{align} \label{e:Cstar1forms}
A_\alpha \,=\, A_\beta - \lambda \cdot dg_{\alpha \beta} g_{\alpha \beta}^{-1}\, ;
\end{align}
since $\C^\times$ is abelian, the conjugation action is trivial. Under the epimorphism in
\eqref{e:ZrelC}, we may consider the $A_\alpha$ as relative $1$-forms. If $d_\C$ denotes the
relative exterior differential (thus treating functions pulled back from $\C$ under $\pi^\circ$ as ``constant''),
then applying it to both sides of \eqref{e:Cstar1forms}, and noting that
\begin{align*}
d_\C \left( \lambda \cdot dg_{\alpha \beta} g_{\alpha \beta}^{-1} \right)
\,=\, \lambda \cdot dg_{\alpha \beta} g_{\alpha \beta}^{-1} \wedge dg_{\alpha \beta} g_{\alpha \beta}^{-1} = 0
\end{align*}
because $dg_{\alpha \beta} g_{\alpha \beta}^{-1}$ is a scalar valued $1$-form, we
obtain a well-defined relative $2$-form $$\omega^\circ \,\in\, H^0( Z^\circ,
\,\Omega_{Z^\circ/\C}^2)$$ such that
\begin{align} \label{e:Phidef}
\omega^\circ|_{Z_\alpha^\circ} \,=\, d_\C A_\alpha\, .
\end{align}

In the case that $\lambda \,=\, 0$, so that $C_0 \,=\, T^*X$, the
equation \eqref{e:Cstar1forms} shows that the 
$A_\alpha$ already patch together to give a well-defined $1$-form $\theta$ on $T^*X$. In fact, 
from their definition \eqref{e:lambdacxnform} one sees that $\theta$ is the tautological 
$1$-form on $T^*X$. Hence $\omega^\circ|_{T^*X} \,= \,d\theta$ is the standard (holomorphic) 
symplectic form on $T^*X$.

Furthermore, since the $A_\alpha$ are locally defined from the tautological sections of $T^* X_\alpha$, we see that $\omega^\circ$ restricted to any fibre is a holomorphic symplectic form. Restricting to any fibre $C_\lambda$, we see that
\begin{align*}
(d_\C A_\alpha)|_{C_{\lambda,\alpha}} \,=\, d(A_\alpha|_{C_\lambda})\, .
\end{align*}
In particular, for $\lambda \,=\, 1$, the $A_\alpha|_C$ give the universal connection, and so $d(A_\alpha|_C)$ is the (holomorphic) curvature of the universal pullback connection.

The foregoing justifies the following proposition.

\begin{prop} \label{p:Cstarrelform}
In the case $G \,= \,\C^\times$, there exists a holomorphic relative $2$-form
$$\omega^\circ
\,\in\, H^0( Z^\circ,\, \Omega_{Z^\circ/\C}^2)$$ which restricts to a holomorphic symplectic form
on each fibre $C_\lambda$ and is the standard holomorphic symplectic form on $C_0 \,=\, T^*X$ and
is the curvature of the ($1$-)connection they define on $C_1$.
\end{prop}

The following will be used later.

\begin{lem} \label{l:gammapullback}
Consider the isomorphism $\gamma \,:\, Z^\times\,\longrightarrow\, C \times \C^\times$ of
Lemma \ref{t:Zstariso} in the case $P$ is a $\C^\times$-bundle. Then if $p \,:\,
C \times \C^\times \,\longrightarrow\, C$ is the projection, and $\omega_1^\circ \,:=\,
 \omega^\circ|_C$, then
\begin{align} \label{e:gammapullback}
\omega^\circ \,=\, \gamma^*(\lambda \cdot p^* \omega_1^\circ)
\,=\, \lambda \cdot \gamma^* p^* \omega_1^\circ\, .
\end{align}
\end{lem}

\begin{proof}
By definition \eqref{e:lambdacxnform}, \eqref{e:Phidef}, $\omega^\circ$ is locally defined as the exterior derivative of the tautological $1$-form on the cotangent bundle. Also by definition \eqref{e:gammadef}, $\gamma$ is defined by multiplying (the fibre component) by $\lambda^{-1}$ to move an arbitrary element of $Z^\times$ into $C$. Thus, the factor of $\lambda$ in \eqref{e:gammapullback} is there to cancel this out.
\end{proof}

\section{Atiyah sequences for principal bundles over homogeneous spaces} \label{s:Atseqhomogsp}

The goal in this section is to show that when $X$ is a homogeneous space, many of the vector bundle constructions in the previous section arise from bundles naturally associated to representations of the relevant groups. We begin with some results holding for general homogeneous spaces and in the latter part of the section, we will specialize to the case of (partial) flag varieties.

\subsection{Description of the universal $\lambda$-connection space} \label{ss:univconnsp}

Let $G$ be a complex connected algebraic group with Lie algebra $\g$, and let $$H 
\,\leq\, G$$ be a closed connected subgroup of $G$ with Lie algebra $\h$;
we will employ the notation ``$\leq$'' to denote a subgroup. 
Set $X \,:= \, 
G/H$. The quotient map $G\, \longrightarrow\, X$ is a holomorphic principal 
$H$-bundle; we will often write $G_H$ when $G$ is considered as this principal 
$H$-bundle.

\begin{lem}\label{le1}
The Atiyah sequence for $G_H$ is the short exact sequence of vector bundles 
on $X$ associated to the short exact sequence of $H$-representations
\begin{align} \label{D}
0 \,\longrightarrow\, \h \,\longrightarrow\, \g \,\stackrel{q}{\longrightarrow}\,
\g / \h \,\longrightarrow\, 0\, .
\end{align}
In particular, $\Theta_X \,=\, G_H \times^H (\g/\h)$. More generally, suppose $M$ is another algebraic group and $\tau \,:\, H \,\longrightarrow\, M$ is a holomorphic homomorphism. Let $G^\tau \,:=\, G_H\times^\tau M\,\longrightarrow\, X$ be the principal $M$-bundle obtained by extending the structure group of $G_H$ using $\tau$. Then the Atiyah sequence of $G^\tau$ is associated to the sequence of $H$-representations
\begin{align} \label{e:AtSeqM}
0 \,\longrightarrow\, \m \,\longrightarrow\, (\g \oplus \m) / \h \,\longrightarrow\, \g/\h
\,\longrightarrow\, 0\, .
\end{align}
Here $\m$ is an $H$-representation via $H \xrightarrow{\tau} M \xrightarrow{\Ad} {\rm GL}(\m)$, the inclusion $\h \hookrightarrow \g \oplus \m$ is
\begin{align} \label{e:hinc}
\xi \,\longmapsto\, \left( \xi,\, - d\tau(\xi) \right)\, ,
\end{align}
and the inclusion of $\m$ in $(\g \oplus \m) / \h$ is given by the natural inclusion of $\m$ in $\g \oplus \m$ followed by the quotient by $\h$. 
\end{lem}

\begin{proof}
We identify $\Theta_G \,=\, G \times \g$ via $(g,\, \xi) \,\longmapsto\, dL_g \xi \in T_g G$, where
$L_g \,:\, G \,\longrightarrow\, G$ is left multiplication by $g \,\in\, G$. Then it is
straightforward to check that the trivial sub-bundle $G \times \h \,\hookrightarrow\, G \times \g$
on $G$ is precisely the kernel of the differential $d\pi_X \,:\, \Theta_G \,\longrightarrow\,
\pi^* \Theta_X$ for the projection $\pi_X \,:\, G \,\longrightarrow\, X$. This inclusion
corresponds precisely to the inclusion of $H$-representations $\h \,\hookrightarrow \,\g$. Hence the
sequence of (trivial) vector bundles
\begin{align*}
0 \,\longrightarrow\, \Theta_{G/X} \,\longrightarrow\, \Theta_G \,\longrightarrow\,
\pi^*\Theta_X \,\longrightarrow\, 0
\end{align*}
on $G_H$ corresponds precisely to the sequence \eqref{D} of $H$-representations, but at the same time
is the sequence \eqref{A1} for $G_H$. The Atiyah sequence for $G_H$ is precisely the quotient of
the former by $H$, which is the same as the exact sequence of vector bundles associated to \eqref{D}.

For the more general statement, consider the actions of $H$ and $M$ on $G\times M$ defined by
\begin{align} \label{E}
(a,\,b) \cdot h & \,= \,\left(ah, \,\tau(h)^{-1} b \right) & (a,\,b) \cdot m & \,=\, (a,\, bm)\, ,
\end{align}
where $h \,\in\, H$, $m \,\in\, M$ and $(a,\,b) \,\in \,G \times M$.

It is easy to see that these commute, so we may consider $G \times M$ as a principal $(H \times 
M)$-bundle over $X$. Furthermore, $G^\tau \,=\, (G \times M) / H$, by definition.

In the above, the isomorphism $G_H \times^H \g/\h \,\stackrel{\sim}{\longrightarrow}\, \Theta_X$ written explicitly is
\begin{align*}
[ g,\, \overline{\xi} ] \,\longmapsto\, d\pi( dL_g \xi)\, ,
\end{align*}
where $\overline{\xi}$ denotes the class $\xi + \h$, and $[g, \,\overline{\xi}]$ the class of
$(g,\, \overline{\xi}) \,\in \,G_H \times \g/\h$ in $G_H \times^H \g/\h$. We have
$T_{(a,b)}(G \times M) \,=\, dL_a \g \oplus dR_b \m$; here we use right invariance in the $\m$-factor
since in \eqref{E} we are multiplying on the left; $R_b \,:\, M \,\longrightarrow\, M$ is
right multiplication by $b \in M$. So as above, since $G \times M$ is a principal $H$-bundle over
$G^\tau$,
\begin{align*}
(G \times M) \times^H \big( (\g \oplus \m) /\h \big)\,\stackrel{\sim}{\longrightarrow}\, \Theta_{G^\tau}
\end{align*} 
via
\begin{align} \label{F}
[a,\,b,\, \overline{(\xi, \eta)}] \,\longmapsto\, d\rho (dL_a \xi,\, dR_b \eta)\, ,
\end{align}
where $\rho \,:\, G \times M \,\longrightarrow\, G^\tau$ is the quotient by $H$. Again, the Atiyah
sequence arises from the inclusion of the vertical tangent bundle of $G^\tau$ into the full tangent
bundle. Lifting this to $G \times M$, this comes from the map of trivial bundles
\begin{align*}
(G \times M) \times \m & \,\longrightarrow\, (G \times M) \times \big( (\g \oplus \m) /\h \big) & 
(a,\,b, \,\eta) & \,\longmapsto\, \left( a,\, b,\, \overline{(0, \Ad_b \eta)} \right).
\end{align*}
(The $\Ad_b$ arises because we use $dR_b$ for the $\m$-factor.) Since this commutes with the
$H$-action, it descends to the quotient $G^\tau \times \m \,\longrightarrow\, \Theta_{G^\tau}
\,=\, (G \times M) \times^H \big( (\g \oplus \m) /\h \big)$
\begin{align*}
\left( [a,\,b],\, \eta \right) & \,\longmapsto\, \left[ a,\, b,\, \overline{(0, \Ad_b \eta)} \right]
\end{align*}
using the isomorphism \eqref{F}.

By definition, $\At G^\tau \,=\, \Theta_{G^\tau}/M$ and we have an isomorphism $\Theta_{G^\tau} / M
\,\longrightarrow\, G_H \times^H \big( (\g \oplus \m) /\h \big)$
\begin{align*}
\left[ a,\,b,\, \overline{(\xi, \eta)} \right] \,\longmapsto\, \left[ a,\,
 \overline{(\xi, \Ad_b \eta)} \right]
\end{align*}
with inverse
\begin{align*}
\left[ a,\, \overline{(\xi, \eta)} \right] \,\longmapsto\, \left[ a,\, e,\,
\overline{(\xi, \eta)} \right]\, .
\end{align*}

Now, $\ad G^\tau\, =\, (G^\tau \times \m)/M \,=\, G_H \times^H \m$, and the map $\ad G^\tau
\,\longrightarrow\, \At G^\tau$ under the above isomorphism is
\begin{align*}
[a, \,\eta] \,\longmapsto\, \left[ a,\, \overline{(0, \eta)} \right]\, ,
\end{align*}
which is clearly associated to the map of $H$-representations
\begin{align*}
\m \,\longmapsto\, (\g \oplus \m) /\h. & \qedhere
\end{align*}
\end{proof}

Let $G$, $H$, $\tau \,:\, H \,\longrightarrow\, M$ be as above and let $P := G_H(M)$. We may now give a description of the base of the universal pullback $\lambda$-connection of $P$ in terms of the representations of $H$. We simply parallel the construction of \eqref{e:AtOmega} by tensoring \eqref{e:AtSeqM} with $(\g/\h)^\vee$ to obtain the top row of
\begin{align}\label{e:RepOmegagen}
\vcenter{ \xymatrix{
0 \ar[r] & \Hom(\g/\h, \,\m) \ar[r] \ar@{=}[d] & \Hom \big( \g/\h, \,(\g \oplus \m)/ \h \big) \ar[r]^-{\sigma \circ -} & \End \g/\h \ar[r] & 0 \\
0 \ar[r] & \Hom(\g/\h, \,\m) \ar[r] & \w \ar[r]^-{\sigma \circ -} \ar@{^{(}->}[u] & \C \ar[r] \ar@{^{(}->}[u] & 0 ,
} }
\end{align}
where $\sigma\, :\, (\g \oplus \m)/ \h\, \longrightarrow\, \g/\h$ is the projection;
in the bottom row, $\w$ is defined as the $H$-subrepresentation 
\begin{align*} 
\w \,:=\, \left\{\nu \,\in\, \Hom\big(\g/\h,\, (\g \oplus \m)/\h \big)\,\mid\, \sigma \circ \nu
\in \C \cdot {\rm Id}_{\g/\h} \right\}\, 
\end{align*}
of $\Hom \big( \g/\h, \,(\g \oplus \m)/\h \big)$. Therefore, in our case, the vector bundle $W_P$ in \eqref{e:AtOmega} is the vector bundle $G_H \times^H \w$ over $G/H$ associated to $\w$. For $\lambda \in \C$, we also define the subset
\begin{align} \label{e:alam}
\a_\lambda \,:=\, \left\{\nu \,\in\, \Hom\big(\g/\h,\, (\g \oplus \m)/\h \big)\,\mid\, \sigma \circ \nu
\,=\, {\rm Id}_{\g/\h} \right\}\, 
\end{align}
of $\w$. This is an affine $H$-space modelled on the $H$-module $\Hom(\g/\h, \m)$ invariant under the $H$-action on $\Hom\big(\g/\h,\, (\g \oplus \m)/\h \big)$. The following is clear from the construction just given and the definition \eqref{e:ClamDef}.

\begin{prop} \label{p:Clamhomog}
In the situation where $X \,=\, G/H$ and $P \,= \,G_H(M)$, the base of universal pullback $\lambda$-connection is
\begin{align*}
C_\lambda(P) \,=\, G \times^H \a_\lambda\, .
\end{align*}
\end{prop}

For simplicity of notation, we will assume $\lambda = 1$, and therefore describe the universal pullback connection for $P \,=\, G_H(M)$. We set $\a \,:=\, \a_1$. One will observe that this ``factor'' of $\a$ is precisely what is required to split the appropriate sequence to obtain a connection trivial on the fibres. We may think of $G \times \a$ as the total space of an $H$-bundle over $G \times^H \a$. Recall that the Atiyah sequence for $G_H(M)$ over $G/H$ was given by \eqref{e:AtSeqM}. Pulling this back to $G \times^H \a$, we get the sequence
\begin{align*}
0 \,\longrightarrow\, (G \times \a) \times^H \m \,\longrightarrow\, (G \times \a) \times^H (\g \oplus \m)/\h
\,\longrightarrow\, (G \times \a) \times^H \g/\h \,\longrightarrow\, 0\, ,
\end{align*}
which in our situation is precisely the lower sequence of \eqref{B1} (for us, $f$ is
the projection $G \times^H \a \,\longrightarrow\,G/H$). This has a tautological splitting $t : (G \times \a) \times^H \g/\h \,\longrightarrow\, (G \times \a) \times^H (\g \oplus \m)/\h$, coming from the factor of $\a$ given by
\begin{align} \label{e:canpbcxnhomog}
[g, \,\nu, \,\xi + \h] \,\longmapsto \, \left[g,\, \nu,\, \nu(\xi + \h) \right]\, .
\end{align}
The fact that this gives a splitting comes from the definition of $\a$ in \eqref{e:alam}. Then it is clear from Definition \ref{d:trivfib} that this yields a connection trivial on the fibres of $G \times^H \a \,\longrightarrow\, G/H$.

\subsection{Pullbacks of homogeneous spaces}

Let $G$, $H$, $\tau \,:\, H \,\longrightarrow\, M$ and $P := G_H(M)$ be as above and 
suppose now that $K \,\leq\, H$ is a closed subgroup; as before,
``$\leq$'' denotes a subgroup. We let $Y \,:=\, G/K$, so that 
there is a smooth (i.e., submersive) map $f \,: \,Y \,=\, G/K \,\longrightarrow\, X 
\,=\, G/H$.

\begin{lem}\label{le2}
The pulled back principal $H$-bundle $f^* G_H$ may be canonically identified with the principal $H$-bundle $G_K(H) := G_K \times^K H$ on $Y$ obtained by extending the structure group of the principal $K$-bundle $G_K$ using the inclusion of $K$ in $H$.
\end{lem}

\begin{proof}
We note that the principal $H$-bundle $G_K(H)$ is the quotient of $G\times H$ where two elements $(g_1, \, h_1), \, (g_2, \, h_2)\, \in\, G\times H$ are identified if there is some $k\, \in\, K$ such that $g_2\,=\, g_1k$ and $h_2\,=\, k^{-1}h_1$. The principal $H$-bundle $f^* G_H$ is the subset of $(G/K)\times G$ consisting of all $(g'K,\, g)$ such that $g'H = gH$. Define the map 
\begin{align*}
& \Phi\, :\, G\times H\, \longrightarrow\, (G/K)\times G,\, & (g,\, h)\, & \longmapsto\,
(gK,\, gh).
\end{align*}
Clearly, $\Phi(G\times H)$ is contained in the above subset $f^* G_H\, \subset\, (G/K)\times G$. For any $k\, \in\, K$,
$$
\Phi(gk,\, k^{-1}h)\,=\, (gk,\, gkk^{-1}h)\,=\, (gk,\, gh)\,=\, \Phi(g,\, h)\, ,
$$
that is, $\Phi$ is $K$-invariant. Hence $\Phi$ descends to a map of total spaces
$$
\widehat{\Phi}\, :\, G_K(H) \, \longrightarrow\, f^* G_H\, 
$$
which intertwines the $H$-actions on $G_K(H)$ and $f^*
G_H$. Consequently, $\widehat{\Phi}$ is an isomorphism of principal $H$-bundles.
\end{proof}

\begin{cor} \label{4}
The composition $K\,\leq\, H\,\stackrel{\tau}{\longrightarrow}\,
M$ allows us to form the associated principal $M$-bundle $G_K(M)$ over $Y$. Then one has a canonical identification
\begin{align*}
f^* \left( G_H(M) \right) \,=\, G_K(M)\, .
\end{align*}
In particular, if $\vv$ is an $H$-representation, and hence a $K$-representation, and if $V := G_H \times^H \vv$ is the associated vector bundle on $X$, then
\begin{align*}
f^*V \,\cong\, G_K \times^K \vv\, ,
\end{align*}
with the right side being the vector bundle over $Y$ associated to $\vv$ as a $K$-representation.
\end{cor}

With this, taking the principal $M$-bundle $G_H(M)$ over $X$, and using Lemma \ref{le1} and Corollary \ref{4}, it is not hard to see that the diagram \eqref{B1} is that of vector bundles over $Y = G/K$ associated to the following diagram of $K$-representations
\begin{align}\label{e:B3}
\vcenter{ \xymatrix{
0 \ar[r] & \m \ar[r] \ar@{=}[d] & (\g \oplus \m) / \k \ar[r] \ar@{>>}[d] & \g/\k \ar[r] \ar@{>>}[d] & 0 \\
0 \ar[r] & \m \ar[r] & (\g \oplus \m) /\h \ar[r]_-{\sigma} & \g/\h \ar[r] & 0.
} }
\end{align}

Now, suppose that the bottom row of \eqref{e:B3} splits as a sequence of $K$-modules, i.e., there is a $K$-equivariant $\varphi \,:\, \g/\h \,\longrightarrow\, (\g \oplus \m)/\h$ such that $\sigma \circ \varphi \,=\, {\rm Id}_{\g/\h}$. This in turn 
yields a splitting of the bottom row of \eqref{B1} and hence a connection on $G_K(M) \,\cong\, 
f^*(G_H(M))$ which is trivial on the fibres of $f$. Applying Proposition \ref{prop1}, we get the 
following.

\begin{prop}\label{propr}
If the bottom row of \eqref{e:B3} splits as a sequence of $K$-modules, then there is a canonical morphism 
\begin{align*}
Y \,=\, G/K \,\longrightarrow\, C \left( G_H(M) \right)\, = \, G \times^H \a\, .
\end{align*}
Using the description of Proposition \ref{p:Clamhomog}, the map $G/K \,\longrightarrow\, G \times^H \a$ is explicitly described by
\begin{align} \label{e:GKmap}
gK \,\longmapsto\, [g,\, \varphi]\, .
\end{align}
(One may note that $\varphi$ being $K$-equivariant means $k \cdot \varphi = \varphi$ for all $k \in K$ and hence the right hand side is independent of the choice of coset representative.)
\end{prop}

\begin{proof} 
It is relatively straightforward to check that this map has the property that the pullback of \eqref{e:canpbcxnhomog} yields the splitting $\varphi$, and so it must be this map by the uniqueness in Proposition \ref{prop1}.
\end{proof}

\begin{cor} \label{c:GKmap}
Assume that $G$ is a connected complex affine algebraic group, $H \,\leq\, G$ a
closed subgroup of $G$ and $K \,\leq\, H$ a reductive subgroup. Then for any morphism of affine algebraic groups $\tau \,:\, H \,\longrightarrow\, M$, there is a canonical morphism
\begin{align*}
G/K \,\longrightarrow\, C \left( G_H(M) \right)\, .
\end{align*}
\end{cor}

\begin{proof}
Since $K$ is reductive, any short exact sequence of $K$-modules splits. Hence it follows from Proposition \ref{propr}.
\end{proof}

\subsection{The universal connection space for $\C^\times$-bundles on partial flag varieties} \label{ss:ucfv}

Let $G$ be a complex simple simply connected group, and let $P\,\leq\, G$ be a
parabolic subgroup of $G$. Let $\chi \,:\, P \,\longrightarrow\, \C^\times$ be a character. The unipotent radical of $P$ will be denoted by $U$. Let $L\, \leq\, P$ be a Levi subgroup, meaning the composition $L\, \hookrightarrow\, P\, \longrightarrow\, P/U$ is an isomorphism. We have the Levi factorization $P \,=\, U \rtimes L$. Since $U$ admits no non-trivial characters, the characters of $P$ and $L$ are identified.

The Lie algebras of $P$, $U$ and $L$ will be denoted by $\p$, $\u$ and $\mathfrak l$ 
respectively. We have $\p\,=\,\u\oplus \l$ (this is a direct sum of vector spaces, 
not Lie algebras, since in general $\l$ normalizes, but does not centralize, $\u$). The
homomorphism of Lie algebras $d\chi \,:\, \p \,\longrightarrow\, \C$ vanishes on 
$[\l,\, \l] + \u$; note that this sum is in fact direct.

Let $\Phi_G$ be the root system for $\g$ with respect to which $P$ is a standard parabolic (i.e., with respect to a maximal torus $T$ of a Borel subgroup $B \leq G$ and with $P \geq B$). Then this comes with a choice of positive roots $\Phi_G^+$. We assume that $L$ is chosen so that $\Phi_L$ is a sub-root system of $\Phi_G$. Then if we take $\Psi \,\subseteq \,\Phi_G^+$ to be the subset of roots whose root spaces appear in $\u$, we have 
\begin{align*}
\Phi_G \,=\, \Phi_L \sqcup \Psi \sqcup \Psi^-\, ,
\end{align*}
where $\Psi^- \,:= \,- \Psi$. We will let $\u_-$ be the sum of the root spaces for $\alpha \,\in\, \Psi^-$; then we have a decomposition
\begin{align*}
\g \,=\, \u_- \oplus \l \oplus \u_+\, ,
\end{align*}
where $\u_+ \,:=\, \u$, and we will often use the subscript $+$ when we speak about both $\u \,=\, \u_+$ and $\u_-$. Observe then that $\p_- := \l + \u_-$ is also a sub-algebra of $\g$ and its corresponding subgroup $P_-$ is the opposite parabolic to $P$.

We wish to consider the situation of the preceding subsections in the case $H = P$ and $M = \C^\times$, $\tau = 
\chi \,:\, P \,\longrightarrow\, \C^\times$ a character of $P$.
In this case, the spaces $\w$ and $\a$ have the explicit descriptions
\begin{align*} 
\w & \,=\, \left\{\nu \,\in\, \Hom\big(\g/\p,\, (\g \oplus \C)/\p \big)\,\mid\, \sigma \circ \nu \in \C \cdot {\rm Id}_{\g/\p} \right\}\, \\
\a & \,=\, \left\{\nu \,\in\, \Hom\big(\g/\p,\, (\g \oplus \C)/\p \big)\,\mid\, \sigma \circ \nu \,=\, {\rm Id}_{\g/\p} \right\}\, \subseteq \w
\end{align*}
where $\sigma$ is the natural projection in \eqref{e:B3}, and \eqref{e:RepOmegagen} becomes
\begin{align} \label{e:RepOmega}
\vcenter{ \xymatrix{
0 \ar[r] & \Hom(\g/\p, \,\C) \ar[r] \ar@{=}[d] & \Hom \big( \g/\p, \,(\g \oplus \C)/ \p \big) \ar[r]^-{\sigma \circ -} & \End \g/\p \ar[r] & 0 \\
0 \ar[r] & \Hom(\g/\p, \,\C) \ar[r] & \w \ar[r]^-{\sigma \circ -} \ar@{^{(}->}[u] & \C \ar[r] \ar@{^{(}->}[u] & 0.
} }
\end{align}
Furthermore, Proposition \ref{p:Clamhomog} then tells us that the base of the universal connection for $G_P(\C^\times) = G_P \times^\chi \C^\times$ is
\begin{align*}
G \times^P \a.
\end{align*}

On the other hand, we may consider the pullback of $G_P(\C^\times)$ to $G/L$ via the projection
$G/L \,\longrightarrow\, G/P$. In our situation, \eqref{e:B3} becomes
\begin{align} \label{e:B3GP}
\vcenter{ \xymatrix{
0 \ar[r] & \C \ar[r] \ar@{=}[d] & (\g \oplus \C) / \l \ar[r] \ar@{>>}[d] & \g/\l \ar[r] \ar@{>>}[d] & 0 \\
0 \ar[r] & \C \ar[r] & (\g \oplus \C) /\p \ar[r]_-{\sigma} & \g/\p \ar[r] & 0,
} }
\end{align}
with the associated diagram of vector bundles over $G/L$ being \eqref{B1}. As $L$ is reductive, Corollary \ref{c:GKmap} gives a canonical morphism
\begin{align} \label{e:GmodLcanmap}
G/L \,\longrightarrow\, C_P(\C^\times) \,= \,G \times^P \a\,.
\end{align}

\begin{prop}\label{p:GmodLCP}
In the case that the character $\chi \,:\, P \,\longrightarrow\, \C^\times$ is
anti-dominant (see the definition below), the morphism \eqref{e:GmodLcanmap} is an
isomorphism. Hence $G/L$ may be identified with the base of the universal
pullback connection for $G_P(\C^\times)$, with the pullback connection
arising from the $L$-splitting of the bottom row of \eqref{e:B3GP}.
\end{prop}

Proposition \ref{p:GmodLCP} will be proved towards the end of this section.

\begin{rmk}
Although Proposition \ref{p:GmodLCP} says that the base space of universal connection
is independent of $\chi$, the curvature of the universal connection, which is a
symplectic form on $G/L$, depends on $\chi$.
\end{rmk}

\begin{defn}
We say that a character $\chi \in X^\bullet(T)$ is \emph{strictly anti-dominant for $P$} if it is anti-dominant in the sense that
\begin{align*}
\langle \chi,\, \alpha^\vee \rangle \,\leq\, 0
\end{align*}
for all $\alpha \,\in\, \Phi_G^+$ with the inequality being strict for all $\alpha \, \in\, \Psi$. Of course, one only requires these inequalities to hold for simple roots in the respective sets.
\end{defn}

To prove Proposition \ref{p:GmodLCP}, we first make more explicit the splitting of the bottom sequence in \eqref{e:B3GP} which gives rise to the morphism \eqref{e:GmodLcanmap}, and then look at the action of $U$ on this splitting and on $\a$. The simplest way to describe the $L$-splitting of \eqref{e:B3GP} is by the map $r \,:\, (\g \oplus \C)/\p \,\longrightarrow\, \C$:
\begin{align} \label{e:LsplitP}
(\xi, c) + \p \,\longmapsto\, d\chi(\xi^\l) + c\, ,
\end{align}
where we write $\xi \,=\, \xi^+ + \xi^\l + \xi^-$ with $\xi^{\pm} \in \u_{\pm}$, $\xi^\l \in \l$. It is clear that this $L$-equivariant, as the $L$-action on $\C$ on both sides is the adjoint action which must be trivial. Of course, the induced splitting of the top row of \eqref{e:B3GP}, which yields the Atiyah sequence of the pullback bundle, has the same expression:
\begin{align} \label{e:Lsplit}
(\xi, c) + \l \,\longmapsto \,d\chi(\xi^\l) + c\, .
\end{align}

We would like to view the splitting as a map $\nu_0 \,:\, \g/\p \,\longrightarrow\, (\g \oplus \C)/\p$, so as to view it as an element of $\a$.
To do this, we choose a basis as follows. For $\alpha\,\in\, \Psi^-$, choose a root vector $f_\alpha$ in the corresponding root space $\g_\alpha$. Then $\{ \overline{f}_\alpha \}_{\alpha \in \Psi^-}$ is a basis for $\g/\p \cong \u_-$ (the bar indicates residue modulo $\p$). One can then verify that
\begin{align} \label{nu0}
\nu_0 & \,:\, \g/\p \,\longrightarrow\, (\g \oplus \C)/\p\, & \overline{f}_\alpha \, & \longmapsto\, \overline{(f_\alpha,0)}\, .
\end{align}
gives the other realization of splitting \eqref{e:LsplitP}; of course, $\nu_0 \,\in\, \a$.

\begin{lem} \label{le:F} \mbox{}
\begin{enumerate}
\item[(a)] The element $\nu_0$ in \eqref{nu0} satisfies
$\ell \cdot \nu_0 \,=\, \nu_0$ for every $\ell \in L$, i.e., for all $\overline{f} \,\in\, \g/\p$, 
\begin{align*}
\nu_0(\ell \cdot \overline{f} ) \,=\, \ell \cdot \nu_0(\overline{f})\, .
\end{align*}

\item[(b)] Assuming that $\chi$ is strictly anti-dominant for $P$, given $\nu \,\in\, \a$, there exists a
unique $u \,\in\, U$ such that $u \cdot \nu_0 \,=\, \nu$.
\end{enumerate}
\end{lem}

\begin{proof}
The statement of (a) holds simply because \eqref{e:LsplitP} is $L$-equivariant and hence $\nu_0$ arises from that map.

For the proof of (b) we need to introduce some further notation. For $\alpha \,\in\, \Phi_G^+$, choose root vectors
$e_\alpha \,\in\, \g_\alpha$, $f_\alpha \,\in\, \g_{-\alpha}$: we will set $h_\alpha \,:=\, [e_\alpha,\, f_\alpha]$. Then, upon
scaling one of $e_\alpha$ or $f_\alpha$, $(e_\alpha,\, h_\alpha,\, f_\alpha)$ would form an $\sl_2$-triple, but as it is, this
need not necessarily be the case. What is important is that with the assumption on $\chi$, 
\begin{align} \label{e:walpha}
w_\alpha := d\chi( h_\alpha) \neq 0.
\end{align}
For $i \in \Z$, we set
\begin{align*}
\g_i & := \,\sum_{\textnormal{ht} \, \alpha = i} \g_\alpha, \ i \,\neq\, 0 & \g_0 & :=\, \t\, ,
\end{align*}
where of course, if $\alpha \,\in\, \Phi_G^-$, then $\textnormal{ht} \, \alpha = - \textnormal{ht}(-\alpha)$. We will denote by $m$ the height of the highest (positive) root. Our argument will be by induction on the heights of the roots.

Let
\begin{align*}
E \,:=\, \sum_{\alpha \in \Psi} c_\alpha e_\alpha \in \u
\end{align*}
with $c_\alpha \,\in\, \C$. Let $u \,:=\, \exp E$; recall that $\exp \,:\, \u \,\longrightarrow\, U$ is an isomorphism. We let
\begin{align*}
E_i \,:=\, \sum_{\textnormal{ht} \, \alpha = i} c_\alpha e_\alpha
\end{align*}
be the $\g_i$-component of $E$, so that
\begin{align*}
E \,= \,\sum_{i=1}^m E_i\, .
\end{align*}

Fix $\alpha \in \Psi$ of height $r$. Then we have an expression
\begin{align*}
\Ad_{u^{-1}} f_\alpha = \Ad_{\exp(-E)} f_\alpha = f_\alpha - [E, f_\alpha] + \frac{1}{2!} [ E, [E, f_\alpha]] - \cdots
\end{align*}
which we may write as
\begin{align*}
\Ad_{u^{-1}} f_\alpha \,=\, f_\alpha + \sum_{i=-(r-1)}^m F_{\alpha, i}^u\, ,
\end{align*}
where $F_{\alpha, i}^u \in \g_i$ is the $\g_i$-component of $\Ad_{u^{-1}}$. For example,
\begin{align*}
F_{\alpha, -(r-1)}^u & = - [E_1, f_\alpha] & f_{\alpha, -(r-2)}^u & = - [E_2, f_\alpha]+ \frac{1}{2} [E_1, [E_1, f_\alpha]].
\end{align*}
In fact, it is not hard to see that 
\begin{align} \label{e:frecur}
F_{\alpha, 0}^u = - [ E_r, f_\alpha] + \zeta_\alpha( f_\alpha, E_1, \ldots, E_{r-1})
\end{align}
for some $\t$-valued function $\zeta_\alpha$ of $f_\alpha, E_1, \ldots, E_{r-1}$.

Given an arbitrary $\nu \,:\, \g/ \p \,\longrightarrow\, (\g \oplus \C)/ \p$ such that $\sigma \circ \nu \,=\, {\rm Id}_{\g/\p}$, we wish to show that we can choose $E$ (i.e., the $c_\alpha$) uniquely so that 
\begin{align*}
u \cdot \nu_0 \,=\, \nu\, .
\end{align*}
In terms of the basis $\{\overline{f}_\alpha \}_{\alpha \in \Psi}$ of $\g/\p$, $\nu$ takes the form
\begin{align*}
\nu( \overline{f}_\alpha ) \,=\, \overline{(f_\alpha, z_\alpha)}
\end{align*}
for some $z_\alpha \,\in\, \C$.

Now, since $U \leq \ker \chi$ and $\g_\alpha \,\subseteq\, [\l, \,\l]$ for all $\alpha \,\in\, \Phi_L$, one finds that $d\chi$ vanishes on $\g_\alpha$ for $\alpha \in \Phi_L \cup \Psi$ and hence we obtain the relation in $(\g \oplus \C)/\p$
\begin{align*}
\overline{(\zeta, 0)} \,= \,
\overline{\big(\zeta, - d\chi(\zeta)\big)} \,= \,\overline{(0,0)}\, ,
\end{align*}
for $\zeta \in \g_\alpha$, $\alpha \in \Phi_G^+$, which we will use repeatedly in what follows.

We observe that, modulo $\p$,
\begin{align*}
\overline{\Ad_{u^{-1}} f_\alpha} \,=\, \overline{f}_\alpha + \sum_{i=-(r-1)}^{-1} \overline{F}_{\alpha, i}^u
\end{align*}
so
\begin{align*}
(u \cdot \nu_0)( \overline{f}_\alpha) & = \Ad_u \nu_0 \left( \overline{\Ad_{u^{-1}} f_\alpha } \right) = \Ad_u \nu_0 \left( \overline{f}_\alpha + \sum_{i=-(r-1)}^{-1} \overline{F}_{\alpha, i}^u \right) = \Ad_u \overline{ \left( f_\alpha + \sum_{i=-(r-1)}^{-1} F_{\alpha, i}^u \right) } \\
& = \overline{ \left( \Ad_u \left( \Ad_{u^{-1}} f_\alpha - \sum_{i=0}^m F_{\alpha, i}^u \right), 0 \right) } = \overline{ \left( f_\alpha - \exp_{\ad E} \left( \sum_{i=0}^m F_{\alpha, i}^u \right), 0 \right) }.
\end{align*}
Now, we observe that
\begin{align*}
\exp_{\ad E} \left( \sum_{i=1}^m F_{\alpha, i}^u \right) \,\in\, \bigoplus_{\alpha \in \Phi_G^+} \g_\alpha
\end{align*}
so this simplifies to
\begin{align*}
\overline{ \left( f_\alpha - \exp_{\ad E} F_{\alpha, 0}^u, 0 \right) } & = \overline{ \left( f_\alpha - F_{\alpha,0}^u - [E, F_{\alpha,0}^u] - \frac{1}{2!} [ E, [E, F_{\alpha,0}^u]] - \cdots, 0 \right) } \\
& = \overline{\left( f_\alpha - F_{\alpha,0}^u, 0 \right)} = \overline{\left( f_\alpha + [E_r, f_\alpha] - \zeta_\alpha(f_\alpha, E_1, \ldots, E_{r-1} ), 0 \right) },
\end{align*}
using \eqref{e:frecur} at the end. Now, if $\textnormal{ht} \, \beta = r = \textnormal{ht} \, \alpha$, but $\beta \neq \alpha$, then $\beta - \alpha \not\in \Phi_G$, so $[e_\beta, f_\alpha] = 0$ and therefore
\begin{align*}
[E_r, \,f_\alpha]\, =\, \sum_{\textnormal{ht} \, \beta = r} c_\beta [ e_\beta,\, f_\alpha] \,= \,c_\alpha h_\alpha\, . 
\end{align*}
Thus,
\begin{align*}
(u \cdot \nu_0)( \overline{f}_\alpha) \,=\, \overline{ \left( f_\alpha + 
c_\alpha h_\alpha - \zeta_\alpha( f_\alpha, E_1, \ldots, E_{r-1} ), 0 \right) }
\end{align*}
In the case that $\textnormal{ht} \, \alpha \,=\, 1$, i.e., $\alpha$ is a simple root in $\Psi$, $\zeta_\alpha = 0$ and 
\begin{align*}
(u \cdot \nu_0)( \overline{f}_\alpha) \,=\,
\overline{ \left( f_\alpha + c_\alpha h_\alpha, 0 \right) }
\,=\, \overline{ \left( f_\alpha, c_\alpha d\chi(h_\alpha) \right) }
\,= \,\overline{ \left( f_\alpha, c_\alpha w_\alpha \right) }
\end{align*}
and using \eqref{e:walpha}, we can uniquely solve the equation $w_\alpha c_\alpha = z_\alpha$ for $c_\alpha$.

By induction on $\textnormal{ht} \, \alpha = r$, we may assume that all $c_\beta$ are determined for $\textnormal{ht} \, \beta < r$, and hence $e_1, \ldots, e_{r-1}$ are determined. Then
\begin{align*}
(u \cdot \nu_0)( \overline{f}_\alpha ) \,=\,
 \overline{ \left( f_\alpha, w_\alpha c_\alpha + d\chi( \zeta_\alpha
(f_\alpha, E_1, \ldots, E_{r-1} ) ) \right) }
\end{align*}
and again we can solve 
\begin{align*}
w_\alpha c_\alpha + d\chi( \zeta_\alpha (f_\alpha, E_1, \ldots, E_{r-1} ) ) \,=\, z_\alpha
\end{align*}
uniquely for $c_\alpha$.
\end{proof}

\begin{rmk} \label{r:unu}
In the proof of Lemma \ref{le:F}(b), we took an arbitrary element $\nu$ with the property that
\begin{align*}
\nu( \overline{f}_\alpha) \,=\, \overline{(f_\alpha, z_\alpha)}\, .
\end{align*}
We may rewrite $\nu = \nu_0 + \varphi$ with $\varphi \in (\g/\p)^\vee$ such that
\begin{align*}
\varphi( \overline{f}_\alpha) \,=\, z_\alpha\, .
\end{align*}

In Section \ref{s:twspGP}, it will be necessary for us to understand what happens when we multiply $\varphi$ by a scalar $\mu$. At the beginning of the induction, we needed to solve the equation $w_\alpha c_\alpha \,=\, z_\alpha$ for $c_\alpha$. Replacing $\varphi$ by $\mu \varphi$ means replacing $z_\alpha$ with $\mu z_\alpha$ and hence our new solution would be $\mu c_\alpha$ instead of $c_\alpha$.

Inductively, one wanted to solve the equation 
\begin{align*}
w_\alpha c_\alpha + d\chi \big( \zeta_\alpha(f_\alpha, E_1, \ldots, E_{r-1}) \big) = z_\alpha,
\end{align*}
where $\alpha$ is a root of height $r$. Replacing $\varphi$ by $\mu \varphi$, since
the expressions preceding \eqref{e:frecur} involve commutators of root
vectors for roots of height $\,<\, r$, one sees that
$$d\chi(\zeta_\alpha(f_\alpha, E_1, \ldots, E_{r-1}))$$ will be a
non-constant polynomial in $\mu$, with coefficients depending on the
$c_\beta$ previously found. But then, the equation above shows that the same
will be true for $c_\alpha$.
\end{rmk}

\begin{proof}[Proof of Proposition \ref{p:GmodLCP}]
The morphism $G/L \,\longrightarrow\, G \times^P \a$ is given in \eqref{e:GKmap} as
\begin{align} \label{e:GmodLiso}
gL \,\longmapsto\, [g,\, \nu_0]\, .
\end{align}
The inverse $G \times^P \a \,\longrightarrow\, G/L$ is given by
\begin{align*}
[g, \,\nu] \,\longmapsto\, guL\, ,
\end{align*}
where $u \in U$ is (the unique) such that $u \cdot \nu_0 = \nu$ (this is
the statement of Lemma \ref{le:F}(b)). It is straightforward to check that this is indeed well-defined and gives the inverse.
\end{proof}

We now record a computational result that will be used in the construction of the
twistor lines in Section \ref{s:twspGP}.

\begin{lem} \label{l:uactionnu0}
The action of the Lie algebra $\u$ on $\nu_0$ is given by
\begin{align*}
e_\alpha \cdot \nu_0 \,=\, s_\alpha e_\alpha
\end{align*}
for some non-zero constants $s_\alpha \in \C^\times$.
\end{lem}

\begin{proof}
Observe that
\begin{align} \label{e:uactionnu}
(e_\alpha \cdot \nu_0)( \overline{f}_\beta)\,=\,
 e_\alpha \cdot \nu_0(\overline{f}_\beta) - \nu_0( \overline{[e_\alpha, f_\beta]}) 
\,=\, \overline{( [e_\alpha, f_\beta], 0)} - \nu_0(\overline{[e_\alpha, f_\beta]})\, .
\end{align}
We consider cases. If $\alpha \,=\, \beta$, $[e_\alpha, f_\alpha]\, =\, h_\alpha \,\in\, \t$ and so then the second term vanishes and the first is
\begin{align*}
\overline{(h_\alpha, 0)} \,=\, \overline{\big(0, d\chi(h_\alpha) \big)} \,=\, d\chi(h_\alpha)
\overline{f}_\alpha^*(\overline{f}_\alpha).
\end{align*}
If $\beta - \alpha \in \Phi_G^+$, then $[e_\alpha, f_\beta] = N_{\alpha \beta} f_{\beta - \alpha}$ for some constant $N_{\alpha \beta}$, so the two terms in \eqref{e:uactionnu} cancel each other out. Finally, if $\beta - \alpha \not\in \Phi_G^+$, then $[e_\alpha, f_\beta] = 0$, so both terms in \eqref{e:uactionnu} are zero.
\end{proof}

\section{Twistor spaces} \label{s:twistor}

\subsection{Real structures on complexifications of (partial) flag varieties} \label{s:realstr}

We will use the notation set at the beginning of Section \ref{ss:ucfv}. The real
structure for our twistor space will come from a compact real form on the Lie algebra
of $G$, so therefore we will record some notation and facts that we will need. Let $d\tau_G \,:\, \g \,\longrightarrow\, \g$ be a compact real form of $\g$; this will integrate to a conjugate linear involution $\tau_G \,:
\, G \,\longrightarrow\, G$, which is a morphism of the underlying real algebraic groups, with the property that 
\begin{align*}
K := \{ x \in G \, \mid \, \tau_G(x) \,=\, x \}
\end{align*}
is a maximal compact subgroup of $G$; we will write $\k$ for the Lie algebra of $K$.

Since any two compact real forms are related by a conjugation, and the same is true of maximal tori, we may in fact take 
the compact real form $d\tau_G$ to be compatible with the root system in the sense that
\begin{align*}
d\tau_G(\g_\alpha) = \g_{-\alpha}
\end{align*}
for all $\alpha \in \Phi_G$ (see, e.g., \cite[Chapter III, proof of Theorem 6.3]{Helgason}). Therefore, for positive roots $\alpha \in \Phi_G^+$, we may choose root vectors
\begin{align*}
e_\alpha & \in \g_\alpha & f_\alpha & \in \g_{-\alpha} 
\end{align*}
so that
\begin{align*}
d\tau_G(e_\alpha) & = f_\alpha & d\tau_G(f_\alpha) & = e_\alpha.
\end{align*}
Of course, we will have chosen the parabolic subgroup $P \leq G$ to be a standard parabolic for this root system, and we choose a Levi subgroup $L \leq P$ whose Lie algebra $\l$ is a sum of root spaces. Then $\l$ will be $d\tau_G$-invariant and hence $L$ is $\tau_G$-invariant.

Let $F \,:= \,P \cap K$. Since unipotent groups have no non-trivial compact subgroups, $U \cap K = \{ e \}$, and so $F \leq L$ and hence
\begin{align} \label{e:PK}
F = L \cap K.
\end{align}
In fact, $F$ will be a maximal compact subgroup of $L$ and one has
\begin{align*}
X \,=\, G/P \,=\, K/F\, .
\end{align*}
Furthermore, for any $P$-variety $T$, we have an identification
\begin{align} \label{GPKF}
G \times^P T \,\cong\, K \times^F T\, .
\end{align}

We will make use of some further properties of $\tau_G$ later in Section \ref{s:twspGP} that we record here for convenience. One is that it commutes with the exponential map: for $\xi \in \g$, one has
\begin{align}\label{e:exptau}
\tau_G \big( \exp(\xi) \big) \,= \,\exp\big( d\tau_G(\xi) \big)\, .
\end{align}
This follows simply because $\tau_G$ is a homomorphism of the underlying real Lie groups.
The second fact is the following: if $\chi \,:\, L \,\longrightarrow\, \C^\times$ is a character, then $d\chi \,:\, \l \,\longrightarrow\, \C$ satisfies 
\begin{align} \label{e:drealstr}
d\chi\big( d\tau_G(\xi) \big) \,=\, - \overline{ d\chi(\xi) }\, .
\end{align}
This can be justified as follows. Since $K \,\leq\, G$ is compact, $\chi(K)$ must be a compact subgroup of $\C^\times$, so $\chi(K)
\,\leq\, S^1$. As the Lie algebra of $S^1$ is $\sqrt{-1} \R \,\subseteq\, \C$, if we write $\g \,=\, \k \oplus
\sqrt{-1}\k$, then $d\chi(\k) \,\subseteq\, \sqrt{-1}\R$; since $d\chi$ is $\C$--linear,
we have $d\chi(\sqrt{-1}\k) \,\subseteq \,\R$. Now, if $\xi \,=\,
A + \sqrt{-1}B \,\in\, \g$, with $A$, $B \,\in\, \k$, we have
$$
- \overline{d\chi(A+\sqrt{-1}B)} \,=\, - \overline{ \big( d\chi(A) +
\sqrt{-1} d\chi(B) \big)}\,=\,
 d\chi(A) - \sqrt{-1} d\chi(B)
$$
$$
=\, d\chi( A - \sqrt{-1}B) \,=\, d\chi \circ d\tau_G(A + \sqrt{-1}B)\, . 
$$

Using the notation of Section \ref{s:univlamcxn}, and in view of Proposition \ref{p:GmodLCP}, let $C \,:=\, G/L$.

\begin{prop} \label{t:cxfnGmodL}
One has an inclusion $\iota \,:\, X \hookrightarrow C$ as a totally real submanifold. The real structure $\tau_G$ descends to one $\tau_C \,:\, C \,\longrightarrow\, C$ for which $X \,=\,
\iota(X) \,=\, C^{\tau_C}$ is the set of the fixed points of $\tau_C$.
\end{prop}

\begin{proof}
{}From \eqref{e:PK}, the inclusion $K \,\hookrightarrow\, G$ induces an injective map $\iota \,:\, 
X \,\longrightarrow\, C$. Since $K$ is compact, the image is closed and the fact that it is a totally real immersion can be checked infinitesimally at the level of Lie algebras.

The real structure $\tau_C \,:\, C \,\longrightarrow\, C$ is induced from $\tau_G \,:\, G\,\longrightarrow\, G$, namely, $\tau_C(gL) \,:= \,\tau_G(g) L$. It is well-defined precisely because $\tau_G$ is a group homomorphism and $L$ is $\tau_G$-invariant.

Since $\tau_G$ fixes $K$ pointwise, it follows immediately that $\tau_C$ fixes $\iota(X)$ pointwise. Note that $\tau_G$ acts on $\g/\k$ as multiplication by $-1$. This implies that $\iota(X)$ is a connected component of the fixed point locus $C^{\tau_C}$. For $g\, \in\, G$, if $g\, =\, k\exp(\sqrt{-1}v)$ is the Cartan decomposition, where $k\, \in\, K$ and $v\, \in\, \k$, then $\tau_G(g)\,=\, k\exp(-\sqrt{-1}v)$. Therefore, if $\tau_G(g)\,=\, g\ell$, where $\ell\, \in\, L$, then
\begin{equation}\label{z1}
\ell\,=\, \exp(-2\sqrt{-1}v)\, .
\end{equation}
Since the Cartan decomposition of $L$ is the restriction of the Cartan decomposition of $G$ to $L$, from \eqref{z1} it follows that $v\, \in\, \l$, and hence $\exp(\sqrt{-1}v)\, \in\, L$. From this it follows that $C^{\tau_C}\, \subset\,\iota(X)$. Hence we have $C^{\tau_C}\, =\,\iota(X)$.
\end{proof}

\begin{rmk}
$C$ is a good complexification of $X$ in the sense of \cite[p.~69]{Totaro}.
\end{rmk}

\subsection{Generalities for construction of twistor spaces} \label{s:twistorgen}

There are some general remarks in \cite[\S~4]{Simpson} which indicate how to construct a
twistor space for a hyper-K\"ahler metric.

Suppose we are given a complex manifold $Z^\circ$ with a surjective submersion $\pi^\circ \,:
\, Z^\circ \,\longrightarrow\, \C$; we set $Z^\times \,:= \,(\pi^\circ)^{-1}(\C^\times)$. Suppose further that we are given an anti-holomorphic involution $\tau^\circ \,:\, Z^\times
\,\longrightarrow\, Z^\times$ such that
\begin{align} \label{e:Zstarreal}
\vcenter{
\commsq{ Z^\times }{ Z^\times }{ \C^\times }{ \C^\times }{ \tau^\circ }{ \pi^\circ }{ \pi^\circ }{ \sigma^\circ } }
\end{align}
commutes, where $\sigma^\circ \,:\, \C^\times \,\longrightarrow\, \C^\times$ is 
\begin{align*}
\sigma^\circ(\lambda) \,=\, - \overline{\lambda}^{-1}\, .
\end{align*}
Let $Z_i^\circ \,:=\, Z^\circ \times \{ i \}$ for $i \,=\, 0$, $1$ and set $Z \,:=\, (Z_0^\circ \coprod \overline{Z}_1^\circ) /
\sim$, where $[z,\, i] \,\sim\, [\tau^\circ(z),\, 1-i]$ for $z \,\in\, Z_i^\times$. Verification of the following statements is
a straightforward exercise.

\begin{lem} \label{t:Zrealstr}\mbox{}
\begin{enumerate}
\item[(a)] $Z$ is Hausdorff and hence a complex manifold. There is a surjective submersion $\pi \,:\,
 Z \,\longrightarrow\, \P^1$ and an anti-holomorphic involution $\tau \,:\, Z \,\longrightarrow\, Z$ such that
\begin{align} \label{ZP1}
\vcenter{ \commsq{ Z }{ Z }{ \P^1 }{ \P^1 }{ \tau }{ \pi }{ \pi }{ \sigma } }
\end{align}
commutes, where $\sigma \,:\, \P^1 \,\longrightarrow\, \P^1$ is the antipodal map.

\item[(b)] Suppose further that we are given a section $\eta^\circ \,:\, \C
\,\longrightarrow\, Z^\circ$ of $\pi^\circ$, i.e., $\pi^\circ \circ \eta^\circ \,=\,
{\rm Id}_\C$, and such that
\begin{align*}
\xymatrix{
Z^\times \ar[r]^{\tau^\circ} & Z^\times \\
\C^\times \ar[u]^{\eta^\circ} \ar[r]_{\sigma^\circ} & \C^\times \ar[u]_{\eta^\circ} }
\end{align*}
commutes. Then there is a well-defined section $\eta \,:\, \P^1 \,\longrightarrow\, Z$ such that
\begin{align} \label{e:Zsn}
\vcenter{
\xymatrix{
Z \ar[r]^{\tau} & Z \\
\P^1 \ar[u]^\eta \ar[r]_\sigma & \P^1 \ar[u]_\eta }
}
\end{align}
commutes.
\end{enumerate}
\end{lem}

\subsection{Twistor spaces for cotangent bundles of homogeneous spaces} \label{s:twspGP}

We take up again the notation of Section \ref{ss:ucfv} and let $X$ be the homogeneous space $G/P$. We choose a strictly anti-dominant character $\chi$ of $P$, for which there is an associated $\C^\times$-bundle, which puts us in the situation of Section \ref{s:Cstarbdls}. (Note that this is also equivalent to choosing a (very) ample line bundle on $X$, and thus fixing a K\"ahler structure on $X$.) The extension \eqref{e:Cstarext} is given by the sequence of vector bundles associated to the lower sequence of $P$-representations in \eqref{e:RepOmega}. Furthermore, the total space $Z^\circ \,=\, \Tot(W_P)$ (see \eqref{e:Cstarext}) has the description
\begin{align} \label{e:Zcirc}
Z^\circ \,= \, G_P \times^P \w\, .
\end{align}

\begin{thm} \label{t:twspGP}
In this situation, $Z^\circ$ gives one patch of the twistor space for a
hyper-K\"ahler metric on $T^*X$ (or equivalently, on the universal connection
space for the frame bundle of the ample line bundle chosen, which may
be identified with $G/L$ by Proposition \ref{p:GmodLCP}). Furthermore, it is clear from the above description that $Z^\circ$ is algebraic.
\end{thm}

\begin{rmk}
It is known that if $X$ is a K\"ahler manifold, then there is a neighbourhood of the zero section in $T^*X$ on which there exists a hyper-K\"ahler metric \cite[Theorem A]{Feix}, \cite[Theorem 1.1]{Kaledin}. In our situation, where $X = G/P$ as above, the above states that the hyper-K\"ahler metric in fact exists on the entirety of the cotangent bundle, which does not necessarily hold in general (cf.~\cite[Theorem B]{Feix}).
\end{rmk}

We note that the proof of Theorem \ref{t:twspGP} involves applying the fundamental 
theorem \cite[Theorem 3.3]{HKLR} characterizing twistor spaces for hyper-K\"ahler 
metrics. We state it here so that, in terms of proof, what is required of us is 
clear.

\begin{thm}
Let $Z$ be a complex manifold of dimension $2n+1$ with a map $\pi \,: \,Z \,\longrightarrow\, \P^1$ making it into a holomorphic fibre bundle. We further assume that 
\begin{enumerate}
\item[(a)] $Z$ admits a real structure $\tau \,:\, Z \,\longrightarrow\, Z$ inducing the antipodal map on $\P^1$;

\item[(b)] $\pi$ admits a family of holomorphic sections $s_\beta \,: \,\P^1 \,
\longrightarrow\, Z$, often referred to as \emph{twistor lines}, each with normal bundle
$\O_{\P^1}(1)^{\oplus 2n}$;

\item[(c)] there exists relative symplectic form $\omega \in \Gamma(Z, \Omega_{Z/\P^1}^2) \otimes \pi^*\O_{\P^1}(2)$, which is compatible with $\tau$ in the sense that $\tau^* \omega = \overline{\omega}$.
\end{enumerate}
Then $Z$ is in fact the twistor space for a hyper-K\"ahler metric on any of the fibres of $\pi$.
\end{thm}

The rest of Section \ref{s:twspGP} will be spent exhibiting these properties and therefore providing a proof of Theorem \ref{t:twspGP}.

\subsubsection{Construction of the total twistor space and the real structure}

Of course, we want to apply the construction in Section \ref{s:twistorgen} to build the twistor space. We use the description of $Z^\circ$ as in \eqref{e:Zcirc}. Explicitly, one has the submersion $\pi^\circ \,:\, Z^\circ \,\longrightarrow\, \C$ given by
\begin{align*}
Z^\circ \,=\, G_P \times^P \w \,\longrightarrow\, G_P \times^P \C \,\longrightarrow\, \C
\end{align*}
To construct a holomorphic submersion $\pi\, :\, Z \,\longrightarrow\, \P^1$ using the results of Section \ref{s:twistorgen}, we need to construct an anti-holomorphic involution $\tau^\circ \,:\,
 Z^\times \,\longrightarrow\, Z^\times$. For this, we use the isomorphism $\gamma : Z^\times \,\stackrel{\sim}{\longrightarrow}\, C \times \C^\times$ of Lemma \ref{t:Zstariso} and the existence of a real structure $\tau_C$ on 
\begin{align*}
C \,:=\, (\pi^\circ)^{-1}(1) \,\cong\, G/L\, .
\end{align*}
Recall that the isomorphism is given in Proposition \ref{p:GmodLCP} and the the real structure by Proposition \ref{t:cxfnGmodL}. Precisely, we set
\begin{align} \label{e:tauc}
\tau^\circ \,:=\, \gamma^{-1} \circ (\tau_C \times \sigma^\circ) \circ \gamma\, .
\end{align}
The fact that \eqref{e:Zstarreal} commutes then comes from the fact that \eqref{e:ZYC} does. This together with Lemma \ref{t:Zrealstr} yields the space $Z$ together with the real structure over the antipodal map on $\P^1$.

\subsubsection{Construction of the relative holomorphic symplectic form}

Proposition \ref{p:Cstarrelform} gives us a form $\omega^\circ \in \Gamma(Z^\circ, \Omega_{Z^\circ/\C}^2)$. As $Z$ is constructed from glueing $Z^\circ$ to $\overline{Z}^\circ$ via $\tau^\circ$, we would like to see that $\overline{\omega}^\circ \in \Gamma(\overline{Z}^\circ, \Omega_{\overline{Z}^\circ/\C}^2)$ patches with $\omega^\circ$ to give a well-defined section of $\Omega_{Z/\P^1}^2 \otimes \O_{\P^1}(2)$. Recall that we had written $\omega_1^\circ \in \Gamma(C, \Omega_C^2)$ for the restriction of $\omega^\circ$ to $C$ in Lemma \ref{l:gammapullback}; suppose for the moment that
\begin{align} \label{e:tauComega}
\tau_C^* \omega_1^\circ \,=\, -\overline{\omega}_1^\circ\, .
\end{align}

We had written $p \,:\, C \times \C^\times \,\longrightarrow\, C$ for the projection. In the following, we will use Lemma \ref{l:gammapullback} which states that
\begin{align*}
\omega^\circ \,=\, \gamma^* (\lambda p^* \omega_1^\circ)
\end{align*}
as well as the definition \eqref{e:tauc} of $\tau^\circ$ and \eqref{e:tauComega}:
\begin{align*}
(\tau^\circ)^* \overline{\omega}^\circ & = (\gamma^{-1} \circ (\tau_C \times \sigma^\circ) \circ \gamma)^* \overline{\gamma^* (\lambda p^* \omega_1^\circ)} = \gamma^* (\tau_C \times \sigma^\circ)^* ( \overline{\lambda} p^* \overline{\omega}_1^\circ) = \gamma^* ( \lambda^{-1} p^* \omega_1^\circ) \\
& = \lambda^{-2} \omega^\circ,
\end{align*}
noting that in the second to last equality, there are cancelling minus signs, one from \eqref{e:tauComega} and the other from the fact that $(\sigma^\circ)^* \overline{\lambda} = -\lambda^{-1}$. Since $\lambda^{-2}$ is the transition function for $\O(2)$, it follows that $\omega^\circ$ and $\overline{\omega}^\circ$ patch together to give a well-defined $\omega \in \Gamma(Z, \Omega_{Z/\P^1}^2 \otimes \O(2))$. Furthermore, the definition of the real structure $\tau \,:\,
 Z \,\longrightarrow\, Z$ is given in the patches by the identity map $Z^\circ \,\longrightarrow\, \overline{Z}^\circ$, and so it is tautological from our definition that
\begin{align*}
\tau^* \omega \,=\, \overline{\omega}\, ,
\end{align*}
which is the compatibility condition with respect to $\tau$.

\begin{proof}[Proof of \eqref{e:tauComega}]
We proceed as follows. By Proposition \ref{p:Cstarrelform}, the form $\omega_1^\circ$ is
the curvature form of the canonical connection, obtained via Proposition \ref{prop1} on
the pullback $Q$ of the $\C^\times$-bundle on $G/P$ to $G/L$ (via the natural
projection $G/L \,\longrightarrow\, G/P$). Note that by Corollary \ref{4}, we know that $Q$
is the $\C^\times$-bundle $G \times^{L, \chi} \C^\times$
over $G/L$ associated to the character $\chi \,:\, L \,\longrightarrow\, \C^\times$.

We will compute the connection $1$-form $\theta$ of this connection (since $Q$ is a $\C^\times$-bundle, this is a scalar $1$-form) associated to a section of $Q$ over a certain Zariski open set $A \,\subseteq\, G/L$. We then show that
\begin{align} \label{e:lastformula}
\tau_C^* \theta \,=\, - \overline{\theta}
\end{align}
on $A \cap \tau_C(A)$. This implies \eqref{e:tauComega} holds on $A \cap \tau_C(A)$: as there one has
\begin{align*}
\tau_C^* \omega_1^\circ \,=\, \tau_C^* \partial \theta \,=\,
 \tau_C^* d\theta \,=\, d \tau_C^* \theta \,=\, - d \overline{\theta}
\,= \,- \dbar \overline{\theta} \,=\, - \overline{\partial \theta}
\,=\, - \overline{\omega}_1^\circ\, .
\end{align*}
Thus, the difference $\tau_C^* \omega_1^\circ - \overline{\omega}_1^\circ$ is a global section of a locally free sheaf which is supported on a proper Zariski closed subset, hence it must vanish everywhere.

It remains to us to specify $A$, the section of $Q$ over $A$, compute $\theta$ and prove \eqref{e:lastformula}. However, we first record the description of the tangent bundle $\Theta_Q$ of $Q$ given by Lemma \ref{le1} as
\begin{align} \label{e:TQ}
\Theta_Q \,=\, \left( (G \times \C^\times) \times (\g \oplus \C)/\l \right) /L\, .
\end{align}
We will also recall that the connection on $Q$ is the splitting of its Atiyah sequence, which arises from the $L$-splitting of the top sequence of \eqref{e:B3GP}.
Then using \eqref{e:Lsplit}, the connection $1$-form $\theta' \,:\, \Theta_Q \,\longrightarrow\, \C$ can be explicitly written
\begin{align} \label{e:thetaprime}
[x,\, a,\, (\xi,\, c) + \l] \,\longmapsto\, \ d\chi(\xi^\l) + c\, ,
\end{align}
where $\xi^\l$ is the $\l$-component of $\xi$ under the decomposition $\g = \u_+ \oplus \l \oplus \u_-$.

The open set $A \,\subseteq\, G/L$ will be the preimage of the open cell $U_-P \,\subseteq \,G/P$ in the Bruhat decomposition of $G/P$, under the projection $G/L \,\longrightarrow\, G/P$; since $P \,=\, U_+L$, we see that $A\,:=\, U_- U_+ L \,\subseteq\, G/L$ is open. We will work with the section $s \,:\, A \,\longrightarrow\, Q$ of $Q$ over $A$ given by
\begin{align*}
(u,\,v) \,\longmapsto\, [uv,\, 1] \,\in\, G \times^{L, \chi} \C^\times\, .
\end{align*}

We will take $\theta$ as the connection $1$-form over $A$ for the connection described above with respect to the section $s$, i.e., $\theta = s^* \theta'$. In order to compute $\theta$, we will specify coordinates on $A$. Fix an ordering of the roots in $\Psi$. Then we take coordinates $x_\alpha$ on $U_-$, $y_\alpha$ on $U_+$, $\alpha \in \Psi$ as follows. One may write
\begin{align*}
U_- & = \prod_{\alpha \in \Psi} \exp( x_\alpha f_\alpha) & U_+ & = \prod_{\alpha \in \Psi} \exp( y_\alpha e_\alpha) 
\end{align*}
and with this, the coordinate tangent vector $\partial_{x_\alpha}$ is given by
\begin{align*}
\frac{d}{d\epsilon} \bigg|_{\epsilon=0} \left( \prod_{\beta < \alpha} \exp(x_\beta f_\beta) \right) \exp\big( (x_\alpha + \epsilon) f_\alpha \big) \left( \prod_{\beta > \alpha} \exp(x_\beta f_\beta) \right).
\end{align*}
We will write the product as $u_\alpha(\epsilon)$. A similar expression holds for $\partial_{y_\alpha}$.

Now, since the expression in \eqref{e:TQ} uses left-invariance to identify tangent vectors with elements in the Lie algebra, we take 
\begin{align*}
X_\alpha := \frac{d}{d\epsilon} \bigg|_{\epsilon = 0} (uv)^{-1} u_\alpha(\epsilon) v;
\end{align*}
note that $X_\alpha$ will depend on $u$, $v$. We have $u_\alpha(\epsilon) v = uv \exp(\epsilon X_\alpha)$ and so
\begin{align*}
ds(\partial_{x_\alpha}) = \frac{d}{d\epsilon} \bigg|_{\epsilon = 0} \left[ uv\exp(\epsilon X_\alpha), 1 \right] = [uv, 1, (X_\alpha, 0) + \l]
\end{align*}
with the latter the class modulo $L$ in the realization \eqref{e:TQ}. Then
under \eqref{e:thetaprime}, the value of $\theta$ can now be computed as
\begin{align*}
\theta(\partial_{x_\alpha})\, = \,d\chi( X_\alpha^\l)\, .
\end{align*}

If we do a similar calculation for $\partial_{y_\alpha}$, then
\begin{align*}
Y_\alpha := \frac{d}{d\epsilon} \bigg|_{\epsilon = 0} (uv)^{-1} u v_\alpha(\epsilon) \,=\, \frac{d}{d\epsilon} \bigg|_{\epsilon = 0} v^{-1} v_\alpha(\epsilon) \in \u_+,
\end{align*}
since $v, v_\alpha(\epsilon) \in U_+$, hence $Y_\alpha = Y_\alpha^+$ and from \eqref{e:thetaprime} it follows that
\begin{align*}
\theta(\partial_{y_\alpha}) = 0.
\end{align*}
Therefore, we obtain
\begin{align} \label{e:theta}
\theta \,=\, \sum_{\alpha \in \Psi} d\chi( X_{\alpha}^\l) dx_\alpha\, .
\end{align}
One may remark that the coordinates $x_\alpha, y_\alpha$ are ``close'' to being Darboux coordinates, so that the $1$-form should take a form similar to that of the canonical $1$-form on the cotangent bundle; this accords with the description in Section \ref{s:Cstarbdls}.

Now we will prove \eqref{e:lastformula}. Since $\theta$ is a holomorphic
$1$-form, and $\tau_C$ is anti-holomorphic, it follows that the pullback $\tau_C^* \theta$
is anti-holomorphic and therefore is of the form
\begin{align*}
\tau_C^* \theta \,=\, \sum a_\alpha d\overline{x}_\alpha + b_\alpha d\overline{y}_\alpha\, ,
\end{align*}
for some functions $a_\alpha$, $b_\alpha$. Now, 
\begin{align*}
ds \circ d\tau_C( \partial_{\overline{y}_\alpha}) \,=\, \left[ \tau_G(uv), 1, \left( d\tau_G(Y_\alpha), 0 \right), + \l \right]
\end{align*}
and so
\begin{align*}
\tau_C^*\theta( \partial_{\overline{y}_\alpha}) \,=\, d\chi \big( d\tau_G(Y_\alpha)^\l \big)
\end{align*}
but since $d\tau_G(Y_\alpha) \in \u_-$, this vanishes, hence all $b_\alpha\, = \,0$. Similarly,
\begin{align*}
\tau_C^*\theta(\partial_{\overline{x}_\alpha}) \,=\, d\chi \big( d\tau_G(X_\alpha)^\l \big) = d\chi \big( d\tau_G(X_\alpha^\l) \big) = - \overline{d\chi ( X_\alpha^\l )}, 
\end{align*}
with the second equality coming from the fact that $d\tau_G$ preserves $\l$ and the third from \eqref{e:drealstr}. Comparing with \eqref{e:theta}, we see that \eqref{e:lastformula} is proved.
\end{proof}

\subsubsection{Construction of the twistor lines} 

The twistor lines, that is the sections $s_\beta \,:\, \P^1 \,\longrightarrow\, Z$, are described as follows. Observe that the map $\nu_0 \,:\, \g/\p \,\longrightarrow\, (\g \oplus \C)/\p$ defined in \eqref{nu0} defines an $L$-splitting (but not a $P$-splitting!) of the lower sequence of $P$-representations in \eqref{e:RepOmega}: such a splitting is given by $\C \,\longrightarrow\, \w$ simply taking
\begin{align*}
1 \,\longmapsto\, \nu_0\, .
\end{align*}
The fact that this gives a morphism of $L$-modules is given by Lemma \eqref{le:F}(a), which says that $L$ acts trivially on $\nu_0$.

We begin, of course, by defining the twistor lines, and then the next task will be to compute their normal bundles. Since the underlying hyper-K\"ahler manifold is $T^*X$, for each point $\beta := [g, \varphi] \in G \times^P (\g/\p)^\vee = T^*X$, we wish to construct a (real) section $s_\beta \,
:\, \P^1 \,\longrightarrow\, Z$. We note that by \eqref{GPKF}, we may assume that $g
\,=\, k \in K$. Of course, we will use the standard charts $U_0$, $U_1$ on $\P^1$, say with $\lambda$ a coordinate on $U_0$ and $\mu = \lambda^{-1}$ a coordinate on $U_1$. We consider the section over $U_0$ given by
\begin{align} \label{e:sbeta}
\lambda \,\longmapsto\, [k,\, \lambda \nu_0 + \varphi] \,\in\, Z^\circ\, .
\end{align}
We wish to see that this extends to a holomorphic section $s_\beta$. Since $Z$ is constructed by glueing $Z^\circ$ to $\overline{Z}^\circ$ via $\tau^\circ$, defined in \eqref{e:tauc}, we would like to look at the image of \eqref{e:sbeta} under $\tau^\circ$ and verify that it extends to a holomorphic section over all of $\P^1$. Recall that Proposition \ref{p:GmodLCP} gave us isomorphisms
\begin{align*}
C \,\cong\, G/L \,\cong\, G \times^P \a\, ,
\end{align*}
with the second map given by \eqref{e:GmodLiso}
\begin{align*}
gL \,\longmapsto\, [g,\, \nu_0]\, .
\end{align*}
From Proposition \ref{t:cxfnGmodL}(b), the real structure $\tau_C$ under
the identification $C \,=\, G/L$, is $$\tau_C(gL) \,=\, \tau_G(g)L\, .$$ Therefore, identifying $C = G \times^P \a$, we can describe the real structure $\tau_C$ as follows: given $[g,\, \nu]\,\in
\, C$, we take $u \in U$ to be the unique element so that $\nu = u \cdot \nu_0$ (Lemma \ref{le:F}(b)) and then
\begin{align*}
\tau_C ([g, \,\nu] ) \,=\, \tau_C( [g, \,u \cdot \nu_0]) \,=\, \tau_C( [gu,\,
 \cdot \nu_0]) \,=\, [ \tau(g)\tau(u),\, \nu_0]\, .
\end{align*}

Using this, under $\tau^\circ$, \eqref{e:sbeta} maps to 
\begin{align} \label{e:taucirctwistor}
[k,\, \lambda \nu_0 + \varphi] & \stackrel{\gamma}{\longmapsto} \left( [k,\, \nu_0 + \mu \varphi],
\lambda \right) \,=\, \left( [k, \,u(\mu, \varphi) \cdot \nu_0 ], \lambda \right)
\,=\, \left( [k u(\mu, \varphi),\, \nu_0 ], \lambda \right) \nonumber \\
& \stackrel{\tau_C \times \sigma^\circ}{\longmapsto}
\left( [k \tau_G \big( u(\mu, \varphi) \big),\,
\nu_0 ], -\overline{\mu} \right) \stackrel{\gamma^{-1}}{\longmapsto}
\left[ k \tau_G \big(
u(\mu, \varphi) \big),\, - \overline{\mu} \nu_0 \right]\, .
\end{align}
Here we have taken $u \,=\, u(\mu, \varphi) \in U$ so that
\begin{align*}
u(\mu, \varphi) \cdot \nu_0 \,=\, \nu_0 + \mu \varphi\, .
\end{align*}
As in the proof of Lemma \ref{le:F}(b), we may write
\begin{align} \label{e:Emuvarphi}
u(\mu, \varphi) & = \exp\big( E(\mu, \varphi) \big) & E & = \,\sum_{j=1}^m E_j(\mu, \varphi) & E_j & = \,\sum_{\height \alpha = j} c_\alpha( \mu, \varphi) e_\alpha.
\end{align}
By Remark \ref{r:unu}, the $c_\alpha(\mu, \varphi)$ are non-constant polynomials in $\mu$.
Then using \eqref{e:exptau}, we have
\begin{align*}
\tau_G\big( u(\mu, \varphi) \big) \,= \,\exp\big( F( \overline{\mu}, \varphi) \big)
\end{align*}
where $F( \overline{\mu}, \varphi) \,=\, d\tau_G( E( \mu, \varphi) )$; explicitly,
\begin{align} \label{e:Fmuvarphi}
F & := \,\sum_{j=1}^m F_j(\overline{\mu}, \varphi) & F_j & :=\, \sum_{\height \alpha = j} \overline{c_\alpha( \mu, \varphi)} f_\alpha
\end{align}
and these are polynomial functions in $\overline{\mu}$. Rewriting \eqref{e:taucirctwistor}, we find
\begin{align*}
\tau^\circ \big( [k, \lambda \nu_0 + \varphi] \big) \,= \,
\left[ k \exp\big( F( \overline{\mu}, \varphi) \big), - \overline{\mu} \nu_0 \right],
\end{align*}
and recalling that we are using the conjugate complex structure on
$\overline{Z}^\circ$, then the right side is a holomorphic section over
$U_1$. Thus, we have a well-defined
\begin{equation}\label{sb}
s_\beta\, : \, \P^1 \,\longrightarrow\, Z\, .
\end{equation}

\subsubsection{Computation of normal bundles to twistor lines}

We now wish to show that the normal bundles to the sections $s_\beta$ in \eqref{sb} 
are isomorphic to $\O_{\P^1}(1)^{2n}$, where $n \,:=\, \dim G/P \,=\, |\Psi|$. We first 
begin by describing the tangent spaces at points of $Z^\circ$ which is itself a 
quotient of $G \times \w$. Let $(g, \nu) \,\in\, G \times \w$ be a representative of 
$[g, \,\nu] \,\in\, G \times^P \w \,=\, Z^\circ$. The tangent space can be identified
\begin{align*}
T_{(g, \nu)} (G \times \w) \,=\, T_g G \oplus T_\nu \w \,=\, \g \oplus \w
\end{align*}
using left translation in the $\g$-factor in the last equality. Therefore, we have the description of the tangent space
\begin{align*}
T_{[g,\nu]} Z^\circ = T_{[g,\nu]} (G \times^P \w) = (\g \oplus \w)/ \p,
\end{align*}
where we include $d\rho : \p \hookrightarrow \g \oplus \w$ (say, if we let $\rho$ denote the action of $P$ on $G \times \w$) via the infinitesimal action:
\begin{align*}
\xi \,\longmapsto\, (\xi, \,\xi \cdot \nu)\, .
\end{align*}
It is not hard to see that
\begin{align*}
\g \oplus \w \,=\, ( \u_- \oplus 0 ) \oplus d\rho( \p) \oplus
\big( 0 \oplus ( \u_+ \oplus \nu_0 ) \big)
\end{align*}
where we identify $\u_+ \,=\, (\g/\p)^\vee$ via the Killing form.

Therefore, for any point $[g, \nu] \in Z^\circ$, the tangent space $T_{[g, \nu]} Z^\circ$ is spanned by the images of $\u_- \oplus 0$ and $0 \oplus (\u_+ \oplus \nu_0)$. Taking up the notation of (the end of) Section \ref{ss:ucfv}, one can take a frame of $T_{[g, \nu]} Z^\circ$ by
\begin{align} \label{e:Zvfs}
\mathbf{f}_\alpha([g, \nu]) & := \dd [ g \exp( \epsilon f_\alpha), \nu], \ \alpha \in \Psi & \mathbf{\nu}_0([g, \nu]) & := \dd [ g, \nu + \epsilon \nu_0] \nonumber \\
\mathbf{e}_\alpha([g, \nu]) & := \dd [ g, \nu + \epsilon e_\alpha], \ \alpha \in \Psi & 
\end{align}

To prove that the normal bundles to the twistor lines are of the appropriate form, we will use induction arguments similar to those of the proof of Lemma \ref{le:F}. Let us choose a ``decreasing'' ordering of the roots $\alpha \in \Psi$, such that $\alpha_1$ is the highest root and $\height \alpha_i\, >\, \height \alpha_j$ implies $i \,<\, j$. We will shorten
\begin{align*}
e_i & := e_{\alpha_i} & f_i & := f_{\alpha_i}
\end{align*}
and let $\mathbf{e}_i$, $\mathbf{f}_i$ be the corresponding vector fields as in \eqref{e:Zvfs}.

We recall that the normal bundle $\NN_{s_\beta}$ corresponding to a section $s_ \beta \,:\, \P^1
\,\longrightarrow\, Z$ is defined by the exact sequence
\begin{align*}
0 \,\longrightarrow\, \Theta_{\P^1} \,\longrightarrow \,\Theta_Z|_{\P^1} \,\longrightarrow
\, \NN_{s_\beta} \,\longrightarrow\, 0\, ,
\end{align*}
where by restriction, we mean the pullback along $s_\beta$. From the expression of the section in \eqref{e:sbeta}, and the vector fields defined in \eqref{e:Zvfs}, it is clear that, in these local expressions, the inclusion $\Theta_{\P^1} \hookrightarrow \Theta_Z|_{\P^1}$ is given by the inclusion of the vector field $\mathbf{\nu}_0$. Therefore, $\NN_{s_\beta}$ is spanned by the $\mathbf{e}_\alpha|_{\P^1}$, $\mathbf{f}_\alpha|_{\P^1}$.

We now want to fix frames for $\NN_{s_\beta}$ over each of $U_0$ and $U_1$ and compute the transition function. Over $U_0$, we will fix the frame $\mathbf{p}_1, \ldots, \mathbf{p}_{2n}$ by
\begin{align*}
\mathbf{p}_i(\lambda) & := \mathbf{e}_i(\lambda), \quad 1 \leq i \leq n & \mathbf{p}_{n+i}(\lambda) & := \mathbf{f}_{n+1-i}(\lambda), \quad 1 \leq i \leq n;
\end{align*}
over $U_1$, we reverse the order, setting $\mathbf{q}_1, \ldots, \mathbf{q}_{2n}$ to be
\begin{align*}
\mathbf{q}_i(\mu) & := \mathbf{f}_i(\overline{\mu}), \quad 1 \leq i \leq n & \mathbf{q}_{n+i}(\mu) & := {\bf e}_{n+1-i}(\overline{\mu}), \quad 1 \leq i \leq n.
\end{align*}

We want to apply the following to the above frames to achieve our desired conclusion.

\begin{lem}
Let $U_0$, $U_1$ be the standard open covering of $\P^1$ with $\mu$ a coordinate on $U_1$. Let $M$ be a rank $m$ vector bundle over $\P^1$ with frames $s_1, \ldots, s_m$ on $U_0$ and $t_1, \ldots, t_m$ over $U_1$. Suppose that the transition function (i.e., the matrix whose columns are the coordinate vectors of the $s_i$ with respect to the $t_j$) with respect to these frames is of the form
\begin{align*}
\left[ \begin{array}{cccccc}
x_1 \mu & g_{12}(\mu) & g_{13}(\mu) & \cdots & g_{1, m-1}(\mu) & g_{1m}(\mu) \\
0 & x_2 \mu & g_{23}(\mu) & \cdots & g_{2, m-1}(\mu) & g_{2m}(\mu) \\
0 & 0 & x_3 \mu & \cdots & g_{3, m-1}(\mu) & g_{3m}(\mu) \\
\vdots & \vdots & \vdots & \ddots & \vdots & \vdots \\
0 & 0 & 0 & \cdots & x_{m-1} \mu & g_{m-1,m}(\mu) \\
0 & 0 & 0 & \cdots & 0 & x_m \mu 
\end{array}
\right]
\end{align*}
where the $x_i$ are non-zero constants and the $g_{ij}(\mu) \in \mu \C[\mu]$. Then $M \,
\cong \O_{\P^1}(1)^m$.
\end{lem}

We therefore compute the change of frame matrix.

Fix $i \in [1, n]$ and suppose $\height \alpha_i = j$ (the reader should keep in mind here that $i$ is the index for the root, while $j$ and $\ell$ will be indices for the height of $\alpha_i$ and other roots). Then 
\begin{align*}
\mathbf{p}_i(\lambda) & = \mathbf{e}_i([k, \lambda \nu_0 + \varphi]) := \dd [ k, \lambda \nu_0 + \varphi + \epsilon e_i] \stackrel{\gamma}{\longmapsto}
 \dd \left( [ k, \nu_0 + \mu( \varphi + \epsilon e_i)], \lambda \right)
\end{align*}
To continue, we need to find $u(\mu, \varphi, \epsilon)$ as in \eqref{e:Emuvarphi} so that
\begin{align*}
u(\mu, \varphi, \epsilon) \cdot \nu_0 = \nu_0 + \mu( \varphi + \epsilon e_i).
\end{align*}
Such will be of the form
\begin{align*}
u(\mu, \varphi, \epsilon) & = \exp\big( E(\mu, \varphi, \epsilon) \big).
\end{align*}
But if $\height \alpha_i = j$, then following the induction procedure in the proof of Lemma \ref{le:F}, since $\overline{f}_i^* \,=\, t_i e_i$ for some non-zero $t_i \in \C^\times$, we see that we will get 
\begin{align*}
E_\ell(\mu, \varphi, \epsilon) & = E_\ell(\mu, \varphi), \quad \ell < j & E_j(\mu, \varphi, \epsilon) & = E_j(\mu, \varphi) + \epsilon \mu t_i e_i 
\end{align*}
and $E_\ell(\mu, \varphi, \epsilon)$ depends on $\epsilon$ for $\ell > j$. Thus,
\begin{align*}
\mathbf{p}_i(\lambda) & \stackrel{\gamma}{\longmapsto}
\dd \left( \left[ k, \exp\big( E(\mu, \varphi, \epsilon) \big) \cdot \nu_0 \right], \lambda \right) = \dd \left( \left[ k \exp\big( E(\mu, \varphi, \epsilon) \big), \nu_0 \right], \lambda \right) \\
& \stackrel{\tau_C \times \sigma^\circ}{\longmapsto} \dd \left( \left[ k \exp\big( F(\mu, \varphi, \epsilon) \big), \nu_0 \right], -\overline{\mu} \right)
\stackrel{\gamma^{-1}}{\longmapsto}
\dd \left[ k \exp\big( F(\mu, \varphi, \epsilon) \big), -\overline{\mu} \nu_0 \right].
\end{align*}
Here,
\begin{align*}
F(\mu, \varphi, \epsilon) & = \,\sum_{\ell = 1}^m F_\ell(\mu, \varphi, \epsilon) 
\end{align*}
with
\begin{align*}
F_\ell(\mu, \varphi, \epsilon) & = F_\ell(\mu, \varphi), \quad 1 \leq \ell < j & F_j(\mu, \varphi, \epsilon) & = F_j(\mu, \varphi) + \epsilon \overline{\mu} \overline{t}_i f_i \\
F_\ell(\mu, \varphi, \epsilon) & = F_\ell(\mu, \varphi) + \epsilon \overline{\mu} \widetilde{F}_\ell, \quad j < \ell \leq m 
\end{align*}
where the $F_\ell(\mu, \varphi)$ are as in \eqref{e:Fmuvarphi} and the $\widetilde{F}_\ell \in \g_\ell$ are sums of root vectors of weight $\ell$. The factor of $\overline{\mu}$ preceding $\widetilde{F}_\ell$ follows from the same reasoning as in Remark \ref{r:unu}. It follows that
\begin{align*}
\mathbf{p}_i(\lambda) & \stackrel{\tau^0}{\longmapsto} \dd \left[ k \exp\big( F(\mu, \varphi) \big) \exp\left( \epsilon \overline{\mu} \left( \overline{t}_i f_i + \sum_{\ell > j} \widetilde{F}_\ell \right) \right), - \overline{\mu} \nu_0 \right] \\
& = \overline{\mu} \left( t_i \mathbf{q}_i(\mu) + \sum_{j < i} t_j(\mu) \mathbf{q}_j(\mu) \right).
\end{align*}
for some functions $t_j(\mu)$. This shows that the first $n$ sections of the frame give upper triangular transition functions.

Now, we also have
\begin{align*}
\mathbf{p}_{n+i}(\lambda) & = \mathbf{f}_{n+1-i}([k, \lambda \nu_0 + \varphi]) := \dd [ k \exp(\epsilon f_{n+1-i}), \lambda \nu_0 + \varphi] \\
& \stackrel{\gamma}{\longmapsto} \dd \left( [k \exp(\epsilon f_{n+1-i}), \nu_0 + \mu \varphi], \lambda \right) = \dd \left( [k \exp(\epsilon f_{n+1-i}), u(\mu, \varphi) \cdot \nu_0 ], \lambda \right) \\
& = \dd \left( [k \exp(\epsilon f_{n+1-i}) u(\mu, \varphi), \nu_0 ], \lambda \right)
\stackrel{\tau_C \times \sigma^\circ}{\longmapsto} \dd \left( [k \exp(\epsilon e_{n+1-i}) \tau_G(u), \nu_0 ], -\overline{\mu} \right) \\
& = \dd \left( [k\tau_G(u) \exp(\epsilon \Ad_{\tau_G(u)^{-1}} e_{n+1-i}) , \nu_0 ], -\overline{\mu} \right).
\end{align*}
Since $\tau_G(u) \in U_-$, we will get
\begin{align*}
\Ad_{\tau_G(u)^{-1}} e_{n+1-i} \,=\, e_{n+1-i} + \sum_{j > {n+1-i}} \widetilde{e}_j + \widetilde{h} + \sum_j \widetilde{f}_j\, ,
\end{align*}
with $\widetilde{e}_j \in \g_{\alpha_j}$, $\widetilde{h} \in \l$ and $\widetilde{f}_j \in \g_{-\alpha_j}$ with $\alpha_j \in \Psi$. Since $\exp( \epsilon \widetilde{h}) \in L$ and $L$ stabilizes $\nu_0$, the above can be written
\begin{align*}
\mathbf{p}_{n+i}(\lambda) & \longmapsto \dd \left( \left[k \tau_G(u) \exp\left( \epsilon \sum \widetilde{f}_j \right), \exp\left( \epsilon \left( e_{n+1-i} + \sum_{j > n+1-i} \widetilde{e}_j \right) \right) \cdot \nu_0 \right], - \overline{\mu} \right) \\
& \longmapsto \dd \left[k \tau_G(u) \exp\left( \epsilon \sum \widetilde{f}_j \right), - \overline{\mu} \nu_0 - \epsilon \overline{\mu} \left( s_{n+1-i} e_{n+1-i} + \sum_{j > n+1-i} \widetilde{s}_j \widetilde{e}_j \right) \right]\, ,
\end{align*}
using Lemma \ref{l:uactionnu0}, where $s_{n+1-i} = s_{\alpha_{n+1-i}}$ and the $\widetilde{s}_j$ are some constants depending on the other $s_\beta$. This last is 
\begin{align*}
- s_{n+1-i} \overline{\mu} \mathbf{e}_{n+1-i}([k \tau_G(u), - \overline{\mu} \nu_0]) - \overline{\mu} \sum_{j > n+1-i} \widetilde{t}_j \mathbf{e}_j([k \tau_G(u), - \overline{\mu} \nu_0]) +
\end{align*}
But this is of the form
\begin{align*}
-s_i \overline{\mu} \mathbf{q}_{n+i}(\mu) + \overline{\mu} \sum_{j < n+i} t_j(\overline{\mu}) \mathbf{q}_j(\mu)\, ,
\end{align*}
for some functions $t_j(\mu)$. We can therefore apply the lemma
above to conclude that the normal bundles are indeed $\O_{\P^1}(1)^{\oplus 2n}$.

\subsection{Application: Purity of Hodge structures of the affine
varieties $G/L$}

We rely on the following statement found at \cite[Appendix B, Theorem B.1]{HLRV}.

\begin{thm}
Let $X$ be a smooth complex algebraic variety and $f \,:\, X \,\longrightarrow\, \C$ a smooth algebraic morphism, i.e., a surjective submersion. Suppose that $X$ admits a $\C^\times$-action and that $f$ is equivariant with respect to a positive power of the standard $\C^\times$-action on $\C$. We further assume that the fixed point set $X^{\C^\times}$ is complete and that for all $x \in X$, $\lim_{\lambda \to 0} \lambda \cdot x$ exists. Then the mixed Hodge structures on all of the fibres are in fact pure and all isomorphic.
\end{thm}

\begin{cor}\label{corf}
The Hodge structure on the quotient $G/L$ is pure.
\end{cor}

\begin{rmk}
This is a consequence of the more general statement \cite[Proposition 2.2.6]{HLRV}. We observe that the construction of Section \ref{ss:ucfv} gives an algebraic family over $\C$ whose fibre over $0$ is the cotangent bundle of a partial flag variety for a reductive group $G$ and whose generic fibre is a coadjoint orbit for $G$. Of interest would be to generalize this question in the following way. Let $R$ be a truncated polynomial ring, say $R \,=\, \C[t]/(t^m)$ for some $m \in \mathbb{N}$. Then $G \otimes_\C R$ is a group over $R$ and we may consider the Weil restriction $G_R$ back to $\C$: this is an algebraic group over $\C$ with $G_R(\C) \,=\, G(R)$. Coadjoint orbits for this group $G_R$ are the building blocks for certain moduli spaces of meromorphic connections (with irregular singularities) over the projective line. It is conjectured that these moduli spaces have pure cohomology \cite{HWW}; a result similar to Corollary \ref{corf} would perhaps be a step towards a proof.
\end{rmk}

\section*{Acknowledgments}

Discussions on this subject began when we were attending the program ``Geometry, Topology and Dynamics of Moduli Spaces'' in 
August, 2016 at the Institute for Mathematical Sciences, National University of Singapore. We would like to thank the organizers, 
especially Graeme Wilkin, for our respective invitations to the event. IB is supported by a J. C. Bose Fellowship. MLW was 
supported by SFB/TR 45 ``Periods, moduli and arithmetic of algebraic varieties'', subproject M08-10 ``Moduli of vector bundles on 
higher-dimensional varieties''. He is also grateful to Daniel Greb for several conversations clarifying the matter discussed 
here. This work was completed while
both the authors were participating in a conference in the
International Centre for Theoretical Sciences; we thank ICTS for its hospitality.

\end{document}